\documentclass[reqno]{amsart}

\newcommand{\mat}[2]{\ensuremath{{#1}^{#2\times #2}}}

\newcommand{\rmat}[3]{\ensuremath{{#1}}^{#2\times #3}}

\newcommand{\nspace}[2]{\ensuremath{{\mathbb{#1}}^{#2}}}
\newcommand{\Hspace}[1]{\ensuremath{\mathcal{#1}}}
\newtheorem{thm}{Theorem}[section]

\newtheorem{lem}[thm]{Lemma}

\theoremstyle{remark}
\newtheorem{rem}[thm]{Remark}

\renewcommand{\Re}{\operatorname{Re}}

\newcommand{\diag}{\operatorname{Diag}}

\newcommand{\row}{\operatorname{Row}}
\newcommand{\col}{\operatorname{Col}}

\newcommand{\rank}{\operatorname{rank}}
\numberwithin{equation}{section}

\newcommand{\cB}{{\mathcal B}}

\newcommand{\cM}{{\mathcal M}}

\newcommand{\cU}{{\mathcal U}}

\newcommand{\cX}{{\mathcal X}}

\usepackage{amsfonts}
\begin{document}

\title[Rational Cayley inner Herglotz-Agler functions]{Rational
Cayley inner Herglotz-Agler functions: positive-kernel decompositions
and transfer-function realizations}
\author[J.A.~Ball]{Joseph A. Ball}
\address{Department of Mathematics \\ Virginia Tech \\
Blacksburg, VA, 24061} \email{joball@math.vt.edu}
\author[D.S.~Kaliuzhnyi-Verbovetskyi]{Dmitry S.
Kaliuzhnyi-Verbovetskyi}
\thanks{The authors were partially supported by US--Israel BSF grant 2010432.
The second author was also partially supported by NSF grant
DMS-0901628, and wishes to thank the Department of Mathematics at
Virginia Tech for hospitality during his sabbatical visit in
January--May 2013, when a significant part of work on the paper
was done.}

\address{Department of Mathematics \\
Drexel University\\
3141 Chestnut Str.\\
  Philadelphia, PA, 19104}
\email{dmitryk@math.drexel.edu}
\date{}

\begin{abstract}
The Bessmertny\u{\i} class consists of rational matrix-valued
functions of $d$ complex variables representable as the Schur
complement of a block of a linear pencil
$A(z)=z_1A_1+\cdots+z_dA_d$ whose coefficients $A_k$ are
 positive semidefinite matrices. We show that it
coincides with the subclass of rational
functions in the Herglotz--Agler class over the right
poly-halfplane which are homogeneous of degree one and which are Cayley inner. The latter means that such
a function is holomorphic on the right poly-halfplane and takes
skew-Hermitian matrix values on $(i\mathbb{R})^d$, or
equivalently, is the double Cayley transform (over the variables
and over the matrix values) of an inner function on the unit
polydisk.
 Using Agler--Knese's characterization of rational inner Schur--Agler functions on the
 polydisk, extended now to the matrix-valued case, and applying appropriate Cayley transformations, we
 obtain characterizations of matrix-valued rational Cayley inner Herglotz--Agler functions both in the setting
  of the polydisk and of the right poly-halfplane, in terms of transfer-function realizations and in terms
   of positive-kernel decompositions. In particular, we extend Bessmertny\u{\i}'s representation to rational Cayley
   inner Herglotz--Agler functions on the right poly-halfplane, where a linear pencil $A(z)$ is now in the form
   $A(z)=A_0+z_1A_1+\cdots +z_dA_d$ with $A_0$ skew-Hermitian and the other coefficients $A_k$ positive
   semidefinite matrices.
\end{abstract}


\maketitle

\section{Introduction}\label{sec:intro}
In the 1980s, M. F. Bessmertny\u{\i} (see \cite{Bes,Bes1,Bes2,Bes3,Bes4}) studied $n\times n$ matrix-valued
rational functions of $d$ variables which admit a so-called finite-dimensional long resolvent representation,
\begin{equation}\label{eq:long-res} f(z)=A_{11}(z)-A_{12}(z)A_{22}(z)^{-1}A_{21}(z),
\quad z=(z_1,\ldots,z_d)\in\nspace{C}{d}.
\end{equation}
Here \begin{equation}\label{eq:lin-pencil} A(z)=A_0+z_1A_1+\cdots+z_dA_d=\begin{bmatrix} A_{11}(z) & A_{12}(z)\\
A_{21}(z) & A_{22}(z)
\end{bmatrix}
\end{equation}
is a linear $\mat{\mathbb{C}}{(n+m)}$-valued function. He showed
that if no additional restrictions on $f$ are assumed, then such a
representation \eqref{eq:long-res} always exists. If, moreover,
$f$ satisfies an additional condition (a)
$f(z)=\overline{f(\bar{z})}$ (resp., (b) $f(z)=f(z)^\top$, (c)
$f(\lambda z)=\lambda f(z)$,
$\lambda\in\mathbb{C}\setminus\{0\}$), then one can choose the
matrices $A_k$, $k=0,\ldots,d$, to be (a) real (resp., (b)
symmetric, (c) such that $A_0=0$).

A particular role in Bessmertny\u{\i}'s work is played by functions of the form \eqref{eq:long-res} with $A_0=0$
and $\overline{A_k}=A_k^\top=A_k\ge 0$), $k=1,\ldots,d$ (i.e., matrices $A_k$ in \eqref{eq:lin-pencil} are
assumed to be real, symmetric, and positive semidefinite), with motivation coming from electrical engineering. He
proved that such functions form the class (which we denote by
${\mathbb{R}}{\mathcal{B}_d^{n\times n}}$) of
characteristic functions of passive $2n$-poles, where impedances of elements (resistances, capacitances,
inductances, and ideal transformers are allowed) are considered as independent variables. (To put it in a
broader context of multidimensional circuit synthesis, see \cite{B}.) It is easy to see that a function
$f\in\mathbb{R}\mathcal{B}_d^{n\times n}$ satisfies the conditions
\begin{equation}\label{eq:homog}
f(\lambda z)=\lambda f(z),\quad \lambda\in\mathbb{C}\setminus\{0\},
\end{equation}
\begin{equation}\label{eq:holomorphic}
f\ {\rm is\ holomorphic\ on\ } \Pi^d,
\end{equation}
\begin{equation}\label{eq:positive}
f(z)+f(z)^*\ge 0,\quad z\in\Pi^d,
\end{equation}
where $\Pi^d=\{z\in\nspace{C}{d}\colon \Re z_k>0,\ k=1,\ldots,d\}$ is the open right poly-halfplane, and
\begin{equation}\label{eq:real}
f(\bar{z})=f(z)^*=\overline{f(z)},
\end{equation}
where $\bar{z}=(\bar{z}_1,\ldots,\bar{z}_d)\in\nspace{C}{d}$. In other words, the class
$\mathbb{R}\mathcal{B}_d^{n\times n}$ is a subclass of the class   of rational $n\times n$ matrix-valued
homogeneous (of degree 1) positive real functions of $d$ variables, denoted $\mathbb{R}\mathcal{P}_d^{n\times
n}$.

We will also consider here the classes $\mathcal{B}_d^{n\times n}=\mathbb{C}\mathcal{B}_d^{n\times n}$ and
$\mathcal{P}_d^{n\times n}=\mathbb{C}\mathcal{P}_d^{n\times n}$. The first one is obtained if we relax the
condition that matrices $A_k$ in a representation \eqref{eq:long-res}--\eqref{eq:lin-pencil} have real entries
and require just $A_0=0$ and $A_k^*=A_k\ge 0$, $k=1,\ldots,d$, and the second one is obtained if we relax the
condition \eqref{eq:real} and require just
\begin{equation}\label{eq:selfadj}
f(\bar{z})=f(z)^*.
\end{equation}
 Thus we have
$\mathbb{R}\mathcal{B}_d^{n\times n}\subseteq\mathcal{B}_d^{n\times n}$ and $\mathbb{R}\mathcal{P}_d^{n\times
n}\subseteq\mathcal{P}_d^{n\times n}$. We also have $\mathcal{B}_d^{n\times n}\subseteq\mathcal{P}_d^{n\times
n}$.

The case of $d=1$ is not interesting: the classes $\mathbb{R}\mathcal{B}_1^{n\times n}$ and
$\mathbb{R}\mathcal{P}_1^{n\times n}$ (resp., $\mathcal{B}_1^{n\times n}$ and $\mathcal{P}_1^{n\times n}$)
coincide and consist of functions of the form $f(z)=Az$ with a $n\times n$ matrix $A$ satisfying
$\overline{A}=A^\top=A\ge 0$ (resp., $A^*=A\ge 0$). If $d=2$, then we also have the coincidence of the classes:
$\mathbb{R}\mathcal{B}_2^{n\times n}=\mathbb{R}\mathcal{P}_2^{n\times n}$ and $\mathcal{B}_2^{n\times
n}=\mathcal{P}_2^{n\times n}$; the first equality was shown by Bessmertny\u{\i} in \cite{Bes1}, and exactly the
same argument works to show the second equality. The question on whether the inclusions
$\mathbb{R}\mathcal{B}_d^{n\times n}\subseteq\mathbb{R}\mathcal{P}_d^{n\times n}$ and $\mathcal{B}_d^{n\times
n}\subseteq\mathcal{P}_d^{n\times n}$ are proper for $d\ge 3$ is open.

Bessmertny\u{\i} has found some necessary conditions for a function $f$ to belong to the class
$\mathbb{R}\mathcal{B}_d^{n\times n}$, however no necessary and sufficient conditions for that in intrinsic
function-theoretical terms (as opposed to the existence of a certain representation) were established in his
work.

In \cite{K-V}, the classes above were generalized as follows. Let $\Hspace{U}$ be a (complex) Hilbert space. The
class $\mathcal{B}_d(\Hspace{U})$ consists of $L(\Hspace{U})$-valued functions $f$ holomorphic on the domain
$$\Omega_d=\bigcup_{\lambda\in\mathbb{T}}(\lambda\Pi)^d\subset\nspace{C}{d}$$ (here, for a fixed
$\lambda\in\mathbb{T}$, we have $\lambda\Pi=\{\lambda z\colon z\in\Pi\}$) and representable there in the form
\eqref{eq:long-res}--\eqref{eq:lin-pencil} where the operators $A_0=0$ and $A_k\in L(\Hspace{U}\oplus\Hspace{H})$
are positive semidefinite (hence selfadjoint), with some Hilbert space $\Hspace{H}$, $k=1,\ldots,d$. Here we
denote the space of bounded linear operators acting from a Hilbert space $\Hspace{X}$ to a Hilbert space
$\Hspace{Y}$ (resp., to $\Hspace{X}$ itself) by $L(\Hspace{X},\Hspace{Y})$ (resp., by $L(\Hspace{X})$). The class
$\mathcal{P}_d(\Hspace{U})$ consists of $L(\Hspace{U})$-valued functions $f$ holomorphic on $\Omega_d$ and
satisfying \eqref{eq:homog}, \eqref{eq:positive}, and \eqref{eq:selfadj}.

Recall \cite{K-V} that a mapping $\iota\colon \Hspace{U}\to\Hspace{U}$ is called an anti-unitary involution of a
Hilbert space $\Hspace{U}$ if $\iota^2=\iota$ and $\langle\iota u_1,\iota u_2\rangle=\langle u_2, u_1\rangle$ for
any $u_1,u_2\in\Hspace{U}$. Such a mapping is anti-linear and bijective. We say that an operator $T\in
L(\Hspace{U},\Hspace{Y})$ is $(\iota_\Hspace{U},\iota_\Hspace{Y})$-real if $\iota_\Hspace{U}$ and
$\iota_\Hspace{Y}$ are anti-unitary involutions of Hilbert spaces $\Hspace{U}$ and $\Hspace{Y}$ and
$\iota_\Hspace{Y}T=T\iota_\Hspace{U}$. (In the case where $\Hspace{Y}=\Hspace{U}$ and
$\iota_\Hspace{Y}=\iota_\Hspace{U}$, we just say ``$\iota_{\Hspace{U}}$-real".) Let
$\Omega\subseteq\nspace{C}{d}$ be a set invariant under (entrywise) complex conjugation, and let
$\iota_\Hspace{U}$ and $\iota_\Hspace{Y}$ be anti-unitary involutions of Hilbert spaces $\Hspace{U}$ and
$\Hspace{Y}$. We say that a function $f\colon\Omega\to L(\Hspace{U},\Hspace{Y})$ is
$(\iota_\Hspace{U},\iota_\Hspace{Y})$-real if $f^\sharp(z)=f(z)$, $z\in\Omega$, where
$f^\sharp(z)=\iota_\Hspace{Y}f(\bar{z})\iota_\Hspace{U}$. (In the case where $\Hspace{Y}=\Hspace{U}$ and
$\iota_\Hspace{Y}=\iota_\Hspace{U}$, we just say ``$\iota_{\Hspace{U}}$-real".) If $\Hspace{U}=\nspace{C}{n}$,
$\Hspace{Y}=\nspace{C}{m}$ and $\iota_\Hspace{U}$, $\iota_\Hspace{Y}$ are complex conjugations, then the matrix
of a $(\iota_\Hspace{U},\iota_\Hspace{Y})$-real operator $T\in L(\nspace{C}{n},\nspace{C}{m})$ in the standard
bases has all real entries, and a $(\iota_\Hspace{U},\iota_\Hspace{Y})$-real function $f\colon\Omega\to
L(\nspace{C}{n},\nspace{C}{m})$ satisfies $f(\bar{z})=\overline{f(z)}$, $z\in\Omega$ --- we will call such a
function real. The class $\iota\mathbb{R}\mathcal{P}_d(\Hspace{U})$ is a subclass of
$\mathcal{P}_d(\Hspace{U})$ consisting of $\iota$-real functions, where $\iota=\iota_\Hspace{U}$ is an
anti-unitary involution of $\Hspace{U}$. The class
 $\iota\mathbb{R}\mathcal{B}_d(\Hspace{U})$ consists of functions $f$
for which there exist a Hilbert space $\Hspace{H}$, an
anti-unitary involution $\iota_\Hspace{H}$ of $\Hspace{H}$, and a long resolvent representation
\eqref{eq:long-res}--\eqref{eq:lin-pencil} of $f$ such that $A_0=0$ and the operators $A_k\in
L(\Hspace{U}\oplus\Hspace{H})$, $k=1,\ldots,d$, are (selfadjoint) positive semidefinite and $\iota_\Hspace{U}\oplus\iota_\Hspace{H}$-real.

Thus the classes $\mathcal{B}_d(\Hspace{U})$, $\mathcal{P}_d(\Hspace{U})$,
$\iota\mathbb{R}\mathcal{B}_d(\Hspace{U})$, and $\iota\mathbb{R}\mathcal{P}_d(\Hspace{U})$ are generalizations of
the classes $\mathcal{B}_d^{n\times n}$, $\mathcal{P}_d^{n\times n}$, $\mathbb{R}\mathcal{B}_d^{n\times n}$, and
$\mathbb{R}\mathcal{P}_d^{n\times n}$, respectively. For these generalized classes we also have that
$\mathcal{B}_d(\Hspace{U})\subseteq\mathcal{P}_d(\Hspace{U})$,
$\iota\mathbb{R}\mathcal{B}_d(\Hspace{U})\subseteq\iota\mathbb{R}\mathcal{P}_d(\Hspace{U})$ (of course, for the
same $\iota$ in both classes in the last inclusion); if $d=1,2$, then these inclusions are equalities; the class
$\mathcal{B}_1(\Hspace{U})=\mathcal{P}_1(\Hspace{U})$ (resp.,
$\iota\mathbb{R}\mathcal{B}_1(\Hspace{U})=\iota\mathbb{R}\mathcal{P}_1(\Hspace{U})$) consists of functions of the
form $f(z)=Az$ with a positive semidefinite operator $A\in L(\Hspace{U})$ (resp., with a $\iota$-real positive
semidefinite operator $A$); the question on whether the inclusions are proper for $d\ge 3$ is open.

In \cite{K-V}, several characterizations of the classes $\mathcal{B}_d(\Hspace{U})$ and
$\iota\mathbb{R}\mathcal{B}_d(\Hspace{U})$ were obtained via the double Cayley transformation which establishes
the relation of these classes to the Schur--Agler class $\mathcal{SA}_d(\Hspace{U})$. Let
$f\in\mathcal{P}_d(\Hspace{U})$. The double Cayley transform of $f$, denoted  $\mathcal{F}=\mathcal{C}(f)$, is defined as
\begin{equation}\label{eq:double-Cayley}
\mathcal{F}(\zeta)=\left(f\Big(\frac{1+\zeta_1}{1-\zeta_1},\ldots,\frac{1+\zeta_d}{1-\zeta_d}\Big)-I_\Hspace{U}\right)
\left(f\Big(\frac{1+\zeta_1}{1-\zeta_1},\ldots,\frac{1+\zeta_d}{1-\zeta_d}\Big)+I_\Hspace{U}\right)^{-1},\
\zeta\in\nspace{D}{d},
\end{equation}
where $\nspace{D}{d}=\{\zeta\in\nspace{C}{d}\colon |\zeta_k|<1,\ k=1,\ldots,d\}$ is the open unit polydisk. It is
easy to see that the function $\mathcal{F}$ is holomorphic and contractive in $\nspace{D}{d}$ (the latter means that
$\|\mathcal{F}(\zeta)\|\le 1$, $\zeta\in\nspace{D}{d}$), i.e., $\mathcal{F}$ belongs to the $d$-variable Schur class
$\mathcal{S}_d(\Hspace{U})$, and that $\mathcal{F}$ is inner, i.e., the boundary values of $\mathcal{F}$  are unitary operators
almost everywhere on the distinguished boundary $\nspace{T}{d}=\{\zeta\in\nspace{C}{d}\colon |\zeta_k|=1,\
k=1,\ldots,d\}$ of the polydisk $\nspace{D}{d}$. The Schur--Agler class $\mathcal{SA}_d(\Hspace{U})$ is a
subclass of $\mathcal{S}_d(\Hspace{U})$ consisting of functions
$\mathcal{F}(\zeta)=\sum_{t\in\mathbb{Z}^d_+}\widehat{\mathcal{F}}_t\zeta^t$ (here $\zeta^t=\zeta_1^{t_1}\cdots \zeta_d^{t_d}$)
satisfying $\|\mathcal{F}(T)\|\le 1$ for every $T\in\mathcal{C}^d$, the class of $d$-tuples of commuting strict
contractions on a Hilbert space, say $\Hspace{K}$, where $\mathcal{F}(T)=\sum_{t\in\mathbb{Z}^d_+}\widehat{\mathcal{F}}_t\otimes
T^t\in L(\Hspace{U}\otimes\Hspace{K})$ and $T^t=T_1^{t_1}\cdots
T_d^{t_d}$.

 Here we say
that the function $\Theta \colon \Lambda \times \Lambda
\to L(\cU)$ is a {\em positive kernel on a set $\Lambda$} if it
holds that
\begin{equation}   \label{posker}
\sum_{i,j=1}^{N} \langle \Theta(\lambda_{i}, \lambda_{j}) u_{j}, u_{i}
\rangle_{\cU} \ge 0 \text{ for all } \lambda_{1}, \dots, \lambda_{N} \in \Lambda,\ u_{1}, \dots, u_{N} \in \cU,
\end{equation}
for all $N=1,2, \dots$.  An equivalent condition is that there exist
a Hilbert space $\cM$ and a function $\theta \colon \Lambda \to L(\cU, \cM)$
so that
\begin{equation}   \label{Kolmogorovdecom}
    \Theta(\omega, \zeta) = \theta(\omega)^{*} \theta(\zeta) \text{ for all }
    \omega, \zeta \in \Lambda.
\end{equation}

\begin{thm}[\cite{Ag}]\label{thm:Agler}
Let $\mathcal{F}$ be a holomorphic $L(\Hspace{U})$-valued function on $\nspace{D}{d}$. The following statements are
equivalent:
\begin{itemize}
    \item[(1)] $\mathcal{F}\in\mathcal{SA}_d(\Hspace{U})$.
    \item[(2)] There exist  positive
    kernels $\Theta_k(\omega, \zeta)$
    on $\nspace{D}{d}$, $k=1,\ldots,d$,  holomorphic in $\zeta$ and
    anti-holomorphic in $\omega$, such that
    \begin{equation}\label{eq:Agler-decomp}
I_\Hspace{U}-\mathcal{F}(\omega)^*\mathcal{F}(\zeta)=\sum_{k=1}^d
(1-\overline{\omega}_k\zeta_k)\Theta_k(\omega, \zeta),\quad
\omega,\zeta\in\nspace{D}{d}.
    \end{equation}
    \item[(2$^\prime$)] There exist Hilbert spaces $\Hspace{M}_k$  and holomorphic
    $L(\Hspace{U},\Hspace{M}_k)$-valued functions $\theta_k$ on $\nspace{D}{d}$, $k=1,\ldots,d$, such that
  \begin{equation}\label{eq:Agler-decomp'}
I_\Hspace{U}-\mathcal{F}(\omega)^*\mathcal{F}(\zeta)=\sum_{k=1}^d
(1-\overline{\omega}_k\zeta_k)\theta_k(\omega)^*\theta_k(\zeta),\quad
\omega,\zeta\in\nspace{D}{d}.
    \end{equation}
    \item[(3)] There exist Hilbert spaces $\Hspace{X}$, $\Hspace{X}_1$, \ldots,
    $\Hspace{X}_d$ with $\Hspace{X}=\bigoplus_{k=1}^d\Hspace{X}_k$, and a unitary operator
\begin{equation}\label{eq:u}
  U=\begin{bmatrix}
A & B \\
C & D
    \end{bmatrix}\in L(\Hspace{X}\oplus\Hspace{U})
\end{equation}
such that
\end{itemize}
\begin{equation}\label{eq:tf}
\mathcal{F}(\zeta)=D+C(I_{\Hspace{X}}-P(\zeta)A)^{-1}P(\zeta)B, \quad \zeta\in\nspace{D}{d},
\end{equation}
where $P(\zeta)=\zeta_1P_{\Hspace{X}_1}+\cdots +\zeta_dP_{\Hspace{X}_d}$ and $P_{\Hspace{Y}}$ denotes the
orthogonal projector onto a subspace $\Hspace{Y}$ of a Hilbert space $\Hspace{X}$.
\end{thm}
We notice that the representation \eqref{eq:tf} is a realization of $\mathcal{F}$ as the transfer function of a
conservative $d$-dimensional Givone--Roesser system (see details in \cite{BT}).

In order to formulate the main result of \cite{K-V}, we also need the following definitions. The class
$\mathcal{A}^d$ is the class of $d$-tuples $R=(R_1,\ldots,R_d)$ of commuting strictly accretive operators on a
common Hilbert space, say $\Hspace{K}$, i.e., the operators $R_k$ commute and there exists a real constant $s>0$
such that $R_k+R_k^*\ge sI_\Hspace{K}$, $k=1,\ldots, d$. It is easy to see that the operator Cayley transform,
defined by
\begin{equation}\label{eq:op-Cayley}
R_k=(I_\Hspace{K}-T)^{-1}(I_\Hspace{K}+T),\quad k=1,\ldots,d,
\end{equation}
 maps the class $\mathcal{C}^d$ onto the class
$\mathcal{A}^d$, and its inverse
\begin{equation}\label{eq:op-Cayley-inverse}
T_k=(R-I_\Hspace{K})(R+I_\Hspace{K})^{-1},\quad k=1,\ldots,d,
\end{equation}
 maps $\mathcal{A}^d$ onto $\mathcal{C}^d$. For
a function $f\in\mathcal{P}_d(\Hspace{U})$ and an operator
$d$-tuple $R\in\mathcal{A}^d$ we define $f(R)=\mathcal{F}(T)$,
where $\mathcal{F}=\mathcal{C}(f)\in\mathcal{S}_d(\Hspace{U})$ is
given by \eqref{eq:double-Cayley} and $T\in\mathcal{C}^d$ is
defined by \eqref{eq:op-Cayley-inverse}.

\begin{thm}[\cite{K-V}]\label{thm:bessm-class}
Let $f$ be a holomorphic $L(\Hspace{U})$-valued function on
$\Omega_d$. The following statements are equivalent:
\begin{itemize}
    \item[(0)] $f\in\mathcal{B}_d(\Hspace{U})$.
    \item[(1)] $f$ satisfies the conditions:
    \begin{itemize}
        \item[(1a)] $f(\lambda z)=\lambda f(z),\
        \lambda\in\mathbb{C}\backslash\{ 0\},\
        z\in\Omega_d$.
        \item[(1b)] $f(R)+f(R)^*\geq 0,\quad
  R\in\mathcal{A}^d$.
        \item[(1c)] $f(\bar{z})=f(z)^*,\quad z\in\Omega_d$.
    \end{itemize}
   \item[(2)] There exist positive
   kernels $\Phi_k(w,z)$ on
   $\Omega_d$, $k=1, \dots, d$, holomorphic in $z$ and
   anti-holomorphic in $w$, that satisfy
\begin{equation}\label{eq:ker-homogen}
\Phi_k(\lambda w,\lambda z)=\Phi_k(w,z), \quad w,z\in\Omega_d,\
\lambda\in\mathbb{C}\setminus\{0\},
\end{equation}
   $k=1,\ldots,d$, such that
\begin{equation}\label{eq:Bessm-decomp}
f(z)=\sum_{k=1}^dz_k\Phi_k(w,z),\quad w,z\in\Omega_d.
\end{equation}
    \item[(2$^{\prime}$)] There exist Hilbert spaces $\Hspace{M}_k$ and holomorphic
    $L(\Hspace{U},\Hspace{M}_k)$-valued functions $\phi_k$ on $\Omega_d$ that satisfy
\begin{equation}\label{eq:factor-homogen}
\phi_k(\lambda z)=\phi_k(z), \quad z\in\Omega_d,\ \lambda\in\mathbb{C}\setminus\{0\},
\end{equation}
     $k=1,\ldots,d$, such that
\begin{equation}\label{eq:Bessm-decomp'}
f(z)=\sum_{k=1}^dz_k\phi_k(w)^*\phi_k(z),\quad w,z\in\Omega_d.
\end{equation}
   \item[(3)] There exist Hilbert spaces
        $\mathcal{X},\mathcal{X}_1,\ldots,\mathcal{X}_d$ with $\mathcal{X}=\bigoplus_{k=1}^d\mathcal{X}_k$, and a
       representation \eqref{eq:tf}
    of a double Cayley transform  of $f$, $\mathcal{F}=\mathcal{C}(f)$
    (which is defined by \eqref{eq:double-Cayley}),
     such that the operator $U$ in \eqref{eq:u} is not only unitary, but also selfadjoint: $U^{-1}=U^*=U$.
     \end{itemize}
If $\iota=\iota_\Hspace{U}$ is an anti-unitary involution on
$\Hspace{U}$ and $(0)$ is replaced by the condition $(0\iota)$
$f\in\iota\mathbb{R}\mathcal{B}_d(\Hspace{U})$, then one should
add to $(1)$ the condition $(1\iota)$ $f$ is a
$\iota_\Hspace{U}$-real function; add to $(2)$ the condition
$(2\iota)$ $\Phi_k$ are $\iota_\Hspace{U}$-real functions,
$k=1,\ldots,d$; add to $(2^\prime)$ the condition
$(2^\prime\iota)$ $\phi_k$ are
$(\iota_\Hspace{U},\iota_{\Hspace{M}_k})$-real functions for some
anti-unitary involutions $\iota_{\Hspace{M}_k}$, $k=1,\ldots,d$;
and add to $(3)$ the condition $(3\iota)$ $U$ is
$\iota_{\Hspace{X}}\oplus\iota_{\Hspace{U}}$-real for some
anti-unitary involution $\iota_{\Hspace{X}}$ which commutes with
$P_{\Hspace{X}_k}$ for all $k=1,\ldots,d$. Then the modified
conditions $(0)$--$(3)$ are equivalent.
\end{thm}

\begin{rem}  The conventions concerning the definition of a positive
    kernel are actually different in \cite{K-V} from those used here.   Namely, in place of
    the equivalent conditions \eqref{posker} or
    \eqref{Kolmogorovdecom}, the following alternative equivalent
    conditions are used \cite{K-V}:
    \begin{align}
& \sum_{i,j=1}^{N} \langle \Theta(\lambda_{i}, \lambda_{j}) u_{i}, u_{j} \rangle_{\cU} \ge 0
\text{ for all } \lambda_{1}, \dots, \lambda_{N}
\in \Lambda,\ u_{1}, \dots, u_{N} \in \cU,  \label{posker'} \\
&  \Theta(\omega, \zeta) = \theta(\zeta)^{*} \theta(\omega) \text{
for all }
    \omega, \zeta \in \Lambda.
    \label{Kolmogorovdecom'}
\end{align}
That the condition \eqref{posker} or \eqref{Kolmogorovdecom}   is
not equivalent to \eqref{posker'} or \eqref{Kolmogorovdecom'}   in
general (for the matrix-valued case) can be seen as a consequence
of the fact that the matrix transposition map $A \mapsto A^\top$
is not completely positive \cite[Page 144]{Arv}.  However the
analysis in \cite{K-V} was based on the work in \cite{BT} which
used the convention \eqref{posker} or \eqref{Kolmogorovdecom}
rather than \eqref{posker'} or \eqref{Kolmogorovdecom'}.  The
resulting confusion can all be fixed by rearranging the formulas
to conform to consistent conventions.
\end{rem}

Denote by $\mathcal{B}_d^{\rm rat}(\nspace{C}{n})$ (resp., by $\mathbb{R}\mathcal{B}_d^{\rm rat}(\nspace{C}{n})$)
the subclass of $\mathcal{B}_d(\nspace{C}{n})$ (resp., of $\iota\mathbb{R}\mathcal{B}_d^{\rm rat}(\nspace{C}{n})$
with $\iota$ being the complex conjugation operator on $\nspace{C}{n}$) consisting of rational functions. It is
obvious that
\begin{equation}\label{eq:inclusions}
\mathcal{B}_d^{n\times n}\subseteq\mathcal{B}_d^{\rm rat}(\nspace{C}{n}),\quad \mathbb{R}\mathcal{B}_d^{n\times
n}\subseteq\mathbb{R}\mathcal{B}_d^{\rm rat}(\nspace{C}{n}).
\end{equation}

In the present paper, we prove that the inclusions in \eqref{eq:inclusions} are, in fact, equalities.
Moreover, we obtain stronger versions of Theorem \ref{thm:bessm-class} for the classes
$\mathcal{B}_d^{n\times n}=\mathcal{B}_d^{\rm rat}(\nspace{C}{n})$ and $\mathbb{R}\mathcal{B}_d^{n\times
n}=\mathbb{R}\mathcal{B}_d^{\rm rat}(\nspace{C}{n})$ in Sections \ref{sec:B_d} and \ref{sec:RB_d}, respectively. This resolves an
open problem raised in the first two paragraphs on page 257 of \cite{K-V} and
a question in \cite[Problem 2, page 287]{K-V}.

We will say that the  $L({\mathcal U})$-valued function $f$ on $\mathbb{D}^d$ (on $\Pi^d$) is a {\em Cayley inner function} if
$f$ is holomorphic with positive semidefinite real part there and
such that its strong nontangental boundary values $f(t)$ have zero real part:
\begin{equation}   \label{inner1}
  f(t) + f(t)^{*} = 0 \text{\ for\ a.e.\ } t\in \mathbb{T}^d\ (\text{for\ a.e.\ } t \in (i \mathbb{R})^{d}).
\end{equation}
Note that such a function is just the Cayley transform (the double Cayley transform) of an
inner function on the polydisk.\footnote{In the
single-variable case ($d=1$), this is consistent with the terminology
of Rosenblum and Rovnyak \cite{RR}; the parallel engineering
terminology would be (continuous-time, impedance) lossless (see
\cite{Staffans}).}
When $f$ is rational matrix-valued and hence has meromorphic
continuation to ${\mathbb C}^{d}$,  uniqueness of meromorphic
continuation off of $(i{\mathbb R})^{d}$ implies that the condition
\eqref{inner1} can be replaced by
\begin{equation}   \label{inner2}
    f(z) = - f(-\overline{z})^{*} \text{ at all points of analyticity
    of } f.
\end{equation}
Notice that the functions from the class $\mathcal{P}_d(\mathcal{U})$ (and therefore from
 any of the classes $\mathcal{B}_d(\mathcal{U})$, $\mathcal{P}_d^{n\times n}$, $\mathcal{B}_d^{n\times n}$)
  are necessarily Cayley inner on $\Pi^d$.

In Section \ref{S:ratinSA}, we obtain a stronger version of Agler's Theorem \ref{thm:Agler}
 for the class $\mathcal{ISA}^{\rm rat}_d(\mathbb{C}^n)$ of rational inner functions from
 $\mathcal{SA}_d(\mathbb{C}^n)$ as a straightforward extension of Knese's result from \cite{Kn}
 to the matrix-valued case. This result is used then in all subsequent sections. We already
 mentioned characterizations of complex and real rational Bessmertny\u{\i}'s classes that we obtain
 in Sections \ref{sec:B_d} and \ref{sec:RB_d}. In Section \ref{S:CayinD}, we obtain several
  characterizations of the subclass $\mathcal{CIHA}^{\rm rat}(\mathbb{D}^d,\mathbb{C}^n)$ of
  rational Cayley inner functions from the Herglotz--Agler class $\mathcal{HA}(\nspace{D}{d},\nspace{C}{n})$.
   (We recall here that the \emph{Herglotz--Agler class on $\nspace{D}{d}$}, denoted as
   $\mathcal{HA}(\nspace{D}{d},\mathcal{U})$, consists of $L(\mathcal{U})$-valued functions
   which are holomorphic on $\nspace{D}{d}$ and whose values on any commutative $d$-tuple of
   strict contractions on a Hilbert space have positive semidefinite real part.) These results are a
   stronger version of the results in \cite{Ag} where the general Herglotz--Agler class
    $\mathcal{HA}(\nspace{D}{d},\mathcal{U})$ was introduced and characterized.

The \emph{Herglotz--Agler class on $\Pi^d$}, denoted as $\mathcal{HA}(\Pi^d,\mathcal{U})$,
 consists of $L(\mathcal{U})$-valued functions which are holomorphic on $\Pi^d$ and whose
 values on any commutative $d$-tuple of strictly accretive operators on a Hilbert space have
  positive semidefinite real part. We also introduce the subclass $\mathcal{CIHA}(\Pi^d,\mathcal{U})$
  of $\mathcal{HA}(\Pi^d,\mathcal{U})$ that consists of Cayley inner functions. Then it follows from
  Theorem \ref{thm:bessm-class} and a remark two paragraphs above that the class $\mathcal{B}_d(\mathcal{U})$
   is a subclass of functions from $\mathcal{CIHA}(\Pi^d,\mathcal{U})$ satisfying the additional homogeneity
   condition \eqref{eq:homog}, and that $\mathcal{B}_d^{n\times n}=\mathcal{B}_d(\mathbb{C}^n)$ is a subclass
   of rational functions from $\mathcal{CIHA}(\Pi^d,\nspace{C}{n})$ satisfying \eqref{eq:homog}. In Section
   \ref{S:CayinPi}, we obtain several characterizations of the subclass $\mathcal{CIHA}^{\rm rat}(\Pi^d,\nspace{C}{n})$
   of the class $\mathcal{CIHA}(\Pi^d,\nspace{C}{n})$ that consists of rational functions. In particular, we
   extend Bessmertny\u{\i}'s long resolvent representation \eqref{eq:long-res} to
   functions from $\mathcal{CIHA}^{\rm rat}(\Pi^d,\nspace{C}{n})$ where the linear
   pencil \eqref{eq:lin-pencil} has a skew-Hermitian matrix $A_0$ and, as in the case of
    functions from $\mathcal{B}_d^{n\times n}$, the other coefficients $A_k$ are
    positive semidefinite matrices.

We remark that various characterizations of the general Herglotz--Agler
 classes $\mathcal{HA}(\nspace{D}{d},\mathcal{U})$ and $\mathcal{HA}(\Pi^d,\mathcal{U})$
  appear in our paper \cite{BK-V} (we also mention a related recent paper \cite{ADY1}).

\section{The rational inner Schur--Agler class} \label{S:ratinSA}

In this section, we tailor Theorem \ref{thm:Agler} to the case
where $\mathcal{F}$ is finite matrix-valued (so ${\mathcal U} =
{\mathbb C}^{n}$ for some $n \in {\mathbb Z}$) and $\mathcal{F}$
is rational inner, i.e., each matrix entry $\mathcal{F}_{ij}$ of
$\mathcal{F}$ is a rational function and $\mathcal{F}(z)$ is
unitary at each point of analyticity $\omega$ of $\mathcal{F}$ on
the unit torus ${\mathbb T}^{d}$.  We mention that a consequence
of Lemma 6.3 in \cite{BSV} is that the set of singularities of
$\mathcal{F}$ on ${\mathbb T}^{d}$ has ${\mathbb T}^{d}$-Lebesgue
measure zero.  By uniqueness of  analytic continuation (see
\cite[page 21]{Shabat}), we see that the rational matrix function
is inner if and only if the identity
\begin{equation}  \label{inner}
    \mathcal{F}(1/\bar{\omega})^{*} \mathcal{F}(\omega) = I_{n}
\end{equation}
at each nonzero nonsingular point $z$ of $\mathcal{F}$ where $\det \mathcal{F}(z) \ne 0$
(where we set $1/\bar{\omega} = (1/\bar{\omega}_{1}, \dots,
1/\bar{\omega_{d}})$ if
$\omega = (\omega_{1}, \dots, \omega_{d}) \in {\mathbb C}^{d}$).
The following result characterizes the rational inner matrix-valued
Schur--Agler class $\mathcal{ISA}_d^{\rm rat}(\nspace{C}{n})$.  We remark that the single-variable case ($d=1$)
is well known and has origins in the circuit theory literature
(see \cite{AV}), while the bivariate case (where the Schur--Agler class
coincides with the Schur class) seems to have appeared for the first
time in the work of Kummert \cite{Kum} (see \cite{B, Kn, Kn-tri} for
additional discussion),
and the scalar-valued case ($n=1$) for an arbitrary number $d$ of
variables appears in \cite{Kn}.

\begin{thm}  \label{thm:ratinSA} Let $\mathcal{F}$ be a ${\mathbb C}^{n \times
    n}$-valued
    function  of $d$ complex variables.
  The following statements are equivalent:
    \begin{enumerate}
    \item[(1)] $\mathcal{F}\in\mathcal{ISA}_{d}^{\rm rat}({\mathbb C}^{n})$.
        \item[(2$^{\prime}$)] There exist rational ${\mathbb
    C}^{N_{k} \times n}$-valued functions $\theta_{k}$, with some
    $N_{k} \in {\mathbb N}$, $k=1, \dots, d$, which have no singularities on ${\mathbb D}^{d}$ and satisfy
    \eqref{eq:Agler-decomp'}.
        \item[(3)] $\mathcal{F}$ has a finite-dimensional Givone--Roesser unitary
    realization, i.e., there exist $m,m_1,\ldots,m_d\in\mathbb{Z}_+$, with $m=m_1+\cdots+m_d$,
    and a unitary matrix
\begin{equation*}
U=\begin{bmatrix} A & B\\
C & D
\end{bmatrix}\in\mat{C}{(m+n)},
\end{equation*}
where
$$A=[A_{ij}]_{i,j=1,\ldots,d},\quad B=\col_{i=1,\ldots,d}[B_i],\quad C=\row_{j=1,\ldots,d}[C_j]$$
are block matrices with blocks $A_{ij}\in\rmat{\mathbb{C}}{m_i}{m_j}$, $B_i\in\rmat{\mathbb{C}}{m_i}{n}$, and
$C_j\in\rmat{\mathbb{C}}{n}{m_j}$, such that $\mathcal{F}$ has a representation of the form
\begin{equation}\label{eq:tf-fin}
\mathcal{F}(\zeta)=D+C(I_m-P(\zeta)A)^{-1}P(\zeta)B.
\end{equation}
 Here $P(\zeta)=\diag[\zeta_1I_{m_1},\ldots,\zeta_dI_{m_d}]$.
    \end{enumerate}
    \end{thm}
\begin{rem}\label{rem:no_cond_(2)}
We note that the analog of condition (2) in Theorem \ref{thm:Agler}
where the $n\times n$ matrix-valued kernels $\Theta_k(w,\zeta)$ are
assumed to be rational in
$\overline{\omega}=(\overline{\omega}_1,\ldots,\overline{\omega}_d)$
and $\zeta=(\zeta_1,\ldots,\zeta_d)$, does not guarantee that
${\mathcal F}$ is inner, due to fact that $\Theta_{k}$ may fail to
have a Kolmogorov decomposition $\Theta_{k}(\omega, \zeta) =
\theta_{k}(\omega)^{*} \theta_{k}(\zeta)$ as in
condition (2$^{\prime}$) in Theorem \ref{thm:ratinSA} with $\theta_{k}$ rational matrix-valued.
 E.g.,
if $d=1$ and $\mathcal{F}=0$, then the Szeg\H{o} kernel
$\Theta_{\rm Sz}(\omega,\zeta)=\frac{1}{1-\overline{\omega}\zeta}$
has no rational finite matrix-valued Kolmogorov decomposition, while it is rational
in $\overline{w}$ and $\zeta$ and
satisfies \eqref{eq:Agler-decomp}.
\end{rem}

The proof of (3)$\Rightarrow$(1) in Theorem \ref{thm:ratinSA}
follows in the same was as in the bivariate case appearing in
\cite[Theorem 6.1]{BSV}.  Thus, to complete the proof of Theorem
\ref{thm:ratinSA}, it suffices to show
(1)$\Rightarrow$(2$^{\prime}$)$\Rightarrow$(3).  These
implications follow from the following more detailed version of
the result, which is just the matrix-valued extension of Knese's
Theorem 2.9 in \cite{Kn}.

\begin{thm}\label{thm:knese}
Let polynomials $p,q\in \mat{\mathbb{C}}{n}[\zeta_1,\ldots,\zeta_d]$
be given with $p(\zeta)$ invertible
for all $\zeta\in\nspace{D}{d}$.   Consider the following statements:
\begin{enumerate}
    \item[(1Kn)]   $\mathcal{F} := q p^{-1}\in\mathcal{ISA}_d^{\rm rat}(\nspace{C}{n})$.
    \item[(2$^{\prime}$Kn)] There exist $N_k\in\mathbb{N}$ and
$\psi_k\in\rmat{\mathbb{C}}{N_k}{n}[\zeta_1,\ldots,\zeta_d]$, $k=1,\ldots,d$, such that
\begin{equation}\label{eq:Knese}
p(\omega)^*p(\zeta)-q(\omega)^*q(\zeta)=\sum_{k=1}^d(1-\overline{\omega}_k\zeta_k)\psi_k(\omega)^*\psi_k(\zeta).
\end{equation}
\item[(3Kn)] $\mathcal{F}$ has a finite-dimensional Givone--Roesser unitary realization
as in condition (2) of Theorem \ref{thm:ratinSA}.
\end{enumerate}
Then {\rm (1Kn)$\Rightarrow$(2$^{\prime}$Kn)$\Rightarrow$(3Kn)}.
\end{thm}

\begin{proof}
(1Kn)$\Rightarrow$(2$^{\prime}$Kn): By Agler's theorem (see
Theorem \ref{thm:Agler}) there exist Hilbert spaces
 $\Hspace{M}_k$ and
$L(\nspace{C}{n},\Hspace{M}_k)$-valued holomorphic functions $\theta_k$ on $\nspace{D}{d}$, $k=1,\ldots,d$, such
that \eqref{eq:Agler-decomp'} holds. Multiplying both sides of \eqref{eq:Agler-decomp'} by $p(\zeta)$ on the
right and by $p(\omega)^*$ on the left, we obtain
\begin{equation}\label{eq:CD-prelim}
p(\omega)^*p(\zeta)-q(\omega)^*q(\zeta)=\sum_{k=1}^d(1-\overline{\omega}_k\zeta_k)\xi_k(\omega)^*\xi_k(\zeta),\quad
\omega,\zeta\in\nspace{D}{d},
\end{equation}
with $L(\nspace{C}{n},\Hspace{M}_k)$-valued holomorphic functions $\xi_k=\theta_kp$ on $\nspace{D}{d}$. Letting
$\zeta=\omega=t\mu$ where $t\in\mathbb{D}$ and $\mu\in\nspace{T}{d}$, we obtain
\begin{equation}\label{eq:CD-special}
\frac{p(t\mu)^*p(t\mu)-q(t\mu)^*q(t\mu)}{1-|t|^2}=\sum_{k=1}^d\xi_k(t\mu)^*\xi_k(t\mu).
\end{equation}
Since $p(\mu)^*p(\mu)=q(\mu)^*q(\mu)$ for all $\mu\in\nspace{T}{d}$, the numerator of the left-hand side of
\eqref{eq:CD-special} is a polynomial in $t$ and $\bar{t}$ which vanishes on the variety $1-\bar{t}t=0$.
Therefore the left-hand side of \eqref{eq:CD-special} is a polynomial in $t$ and $\bar{t}$, and also a
trigonometric polynomial in $\mu$. We have
$$p(\zeta)=\sum_\alpha p_\alpha \zeta^\alpha,\quad q(z)=\sum_\alpha q_\alpha \zeta^\alpha,\quad
\xi_k(\zeta)=\sum_\alpha \xi_{k,\alpha} \zeta^\alpha, $$ where the first two sums are finite. We also have
$$p(t\mu)^*p(t\mu)=\sum_{\alpha,\beta}p_\beta^*p_\alpha\mu^{\alpha-\beta}\bar{t}^{|\beta|}t^{|\alpha|}$$
(here we use notation $|\alpha|=\alpha_1+\cdots+\alpha_d$),  and similarly for $q$ and $\xi_k$. Therefore, the
$0$-th Fourier coefficients of the two sides of the equality \eqref{eq:CD-special} (as Fourier series in $\mu$)
are
\begin{equation*}
\frac{\sum_\alpha(p_\alpha^*p_\alpha-q_\alpha^*q_\alpha)|t|^{2|\alpha|}}{1-|t|^2}=\sum_{k=1}^d\sum_\alpha
\xi_{k,\alpha}^*\xi_{k,\alpha}|t|^{2|\alpha|}.
\end{equation*}
Since the left-hand side is a polynomial in $|t|^2$ of degree at most $r-1$, where $r$ is the maximum of the
total degrees of $p$ and $q$, so is the right-hand side, i.e., $\xi_{k,\alpha}^*\xi_{k,\alpha}=0$ when
$|\alpha|>r-1$, $k=1,\ldots,d$. This implies that $\xi_{k}$ is a $\Hspace{M}_k$-valued polynomial. We have
$$\xi_k(\omega)^*\xi_k(\zeta)=\sum_{|\alpha|,|\beta|\le r-1}\xi_{k,\beta}^*\xi_{k,\alpha}\bar{\omega}^\beta
\zeta^\alpha,$$ where the  positive semidefinite block matrix $X_k=[\xi_{k,\beta}^*\xi_{k,\alpha}]_{\alpha,\beta}$ (of
size ${ \Big(\begin{matrix} r-1+d\\
d\end{matrix}\Big)}$)  can be factored as $X_k=Y_k^*Y_k$ where $Y_k$ is a matrix of size
$N_k\times {\small\left(\begin{matrix} r-1+d\\
d\end{matrix}\right)}n$ with
\begin{equation}\label{eq:N_k-estimate}
N_k=\rank X_k\le\left(\begin{matrix} r-1+d\\
d\end{matrix}\right)n.
\end{equation}
 Writing $Y_k=\row_{|\alpha|\le r-1}[Y_{k,\alpha}]$, we define
$\psi_k\in\rmat{\mathbb{C}}{N_k}{n}[\zeta_1,\ldots,\zeta_d]$ by
$$\psi_k(\zeta)=\sum_{|\alpha|\le r-1}Y_{k,\alpha}\zeta^\alpha.$$ Then
$$\xi_k(\omega)^*\xi_k(\zeta)=\psi_k(\omega)^* \psi_k(\zeta),\quad \zeta,\omega\in\nspace{C}{d},\quad
k=1,\ldots,d,$$ and \eqref{eq:Knese} holds.

(2$^{\prime}$Kn)$\Rightarrow$(3Kn): We use the so-called lurking
isometry argument. Rearranging the terms in \eqref{eq:Knese}, we
obtain
\begin{equation*}
p(\omega)^*p(\zeta)+\sum_{k=1}^d\overline{\omega}_k\zeta_k\psi_k(\omega)^*\psi_k(\zeta)=
q(\omega)^*q(\zeta)+\sum_{k=1}^d\psi_k(\omega)^*\psi_k(\zeta).
\end{equation*}
Therefore the map
$$\begin{bmatrix}
\zeta_1\psi_1(\zeta)\\
\vdots\\
\zeta_d\psi_d(\zeta)\\
p(\zeta)
\end{bmatrix}h\mapsto
\begin{bmatrix}
\psi_1(\zeta)\\
\vdots\\
\psi_d(\zeta)\\
q(\zeta)
\end{bmatrix}h
$$
is a well-defined linear and isometric map from the span of the elements on the left to the span of the elements
on the right, where both spans are taken over all $\zeta\in\nspace{C}{d}$ and $h\in\nspace{C}{n}$. It may be
extended (if necessary) to a unitary matrix $U$ of the required form where we set $m_k=N_k$, $k=1,\ldots,d$, and
$m=m_1+\cdots+m_d$. Writing $\psi(\zeta)=\col_{k=1,\ldots,d}[\psi_k]$, we have by construction of $U$
\begin{eqnarray*}
AP(\zeta)\psi(\zeta)+Bp(\zeta) &=& \psi(\zeta),\\
CP(\zeta)\psi(\zeta)+Dp(\zeta) &=& q(\zeta).
\end{eqnarray*}
Solving for $\psi(\zeta)$ using the first equation and then plugging the result in the second equation gives
$$\mathcal{F}(\zeta)=q(\zeta)p(\zeta)^{-1}=D+CP(\zeta)(I_m-AP(\zeta))^{-1}B=D+C(I_m-P(\zeta)A)^{-1}P(\zeta)B$$
as desired.
\end{proof}

\section{Characterizations of the class $\mathcal{B}_d^{n\times n}$} \label{sec:B_d}

The following refinement of the main result from \cite{K-V}
identifies the rational subclass $\cB_{d}^{\text{rat}}({\mathbb C}^{n})$ of
the generalized Bessmertny\u{\i} class $\cB_{d}({\mathbb C}^{n})$.

\begin{thm}\label{thm:b} Let $f$ be a $\mathbb{C}^{n\times n}$-valued function of $d$ complex variables.
 The following statements are equivalent:
\begin{itemize}
\item[(0)] $f \in \mathcal{B}_{d}^{n \times n}$.
\item[(1)] $f\in\mathcal{B}_d^{\rm rat}(\nspace{C}{n})$.
\item[(2)] There exist $\mathbb{C}^{n\times n}$-valued functions
$\Phi_k(w,z)$, $k=1,\ldots, d$, which are
rational as functions of $z=(z_1,\ldots,z_d)$ and $\overline{w}=(\overline{w}_1,\ldots,\overline{w}_d)$ and which
are  positive  kernels on
   $\Omega_d$ that satisfy \eqref{eq:ker-homogen}
and \eqref{eq:Bessm-decomp}.
    \item[(2$^\prime$)] There exist rational $\rmat{\mathbb{C}}{N_k}{n}$-valued functions $\phi_k$,
    with some $N_k\in\mathbb{N}$, $k=1,\ldots,d$,
     with no singularities on $\Omega_d$, that satisfy \eqref{eq:factor-homogen} and \eqref{eq:Bessm-decomp'}.
   \item[(3)] $\mathcal{F}=\mathcal{C}(f)$ (see \eqref{eq:double-Cayley}) has a finite-dimensional Givone--Roesser
   representation \eqref{eq:tf-fin} as in part (3) of Theorem \ref{thm:ratinSA} where
   the colligation matrix $U$ has the additional property of being
   Hermitian:
   $$
   U^{-1} = U^{*} = U,
   $$
   and $1\notin\sigma(\mathcal{F}(0))$.
     \end{itemize}
\end{thm}

\begin{proof}[Proof of Theorem \ref{thm:b}]
Our plan is to prove first (2$^\prime$)$\Rightarrow$(2)$\Rightarrow$(1)$\Rightarrow$(2$^\prime$), and then
(2$^{\prime}$)$\Leftrightarrow$(3) and (2$^{\prime}$)$\Leftrightarrow$(0).

(2$^\prime$)$\Rightarrow$(2) is obvious: just define $\Phi_k(w,z)=\phi_k(w)^*\phi_k(z)$, $k=1,\ldots,d$.

(2)$\Rightarrow$(1): It follows from the implication (2)$\Rightarrow$(0) of Theorem \ref{thm:bessm-class}
that $f\in\mathcal{B}_d(\mathbb{C}^n)$. Since the kernels $\Phi_k(w,z)$ are rational matrix-valued functions
in both $\overline{w}$ and $z$, the matrix-valued function $f$ in \eqref{eq:Bessm-decomp} is rational.

(1)$\Rightarrow$(2$^\prime$): If (1) holds, then it follows that $\mathcal{F}=\mathcal{C}(f)$ is a rational
$n\times n$ matrix-valued function in the Schur--Agler class
$\mathcal{SA}_d(\nspace{C}{n})$.  The fact that $f$ is also
Cayley inner then guarantees that $\mathcal{F}$ in addition is inner, i.e.,
$\mathcal{F}\in\mathcal{ISA}_d^{\rm rat}(\nspace{C}{n})$. It follows from implication
(1)$\Rightarrow$(2$^\prime$) of Theorem \ref{thm:ratinSA} that \eqref{eq:Agler-decomp'} holds
with rational $\rmat{\mathbb{C}}{N_k}{n}$-valued functions $\theta_k$, for some
$N_k\in\mathbb{N}$, $k=1,\ldots,d$, which have no singularities on
$\nspace{D}{d}$. Set
\begin{equation}\label{eq:theta_to_phi}
\phi_k(z)=\frac{1}{z_k+1}\theta_k\left(\frac{z_1-1}{z_1+1},\ldots,\frac{z_d-1}{z_d+1}\right)(f(z)+I_n),
\quad k=1,\ldots,d.
\end{equation}
 It is clear that these $\phi_k$ are rational $\rmat{\mathbb{C}}{N_k}{n}$-valued functions
which have no singularities on $\Pi^d$. Moreover, since $f$ is homogeneous of degree 1, the argument
in \cite[Theorem 3.1]{K-V} shows that these functions satisfy \eqref{eq:factor-homogen}
and \eqref{eq:Bessm-decomp'} and by
homogeneity have no singularities on $\Omega_d$. Thus (2$^\prime$) follows.

(2$^{\prime}$)$\Leftrightarrow$(3): Suppose that (2$^\prime$) holds.
As we have already proved, (2$^\prime$)$\Rightarrow$(2)$\Rightarrow$(1), so
$f\in \mathcal{B}_d^{\rm rat}(\mathbb{C}^n)$. Also, we have shown in the preceding paragraph
that $\mathcal{F}=\mathcal{C}(f)\in\mathcal{ISA}_d^{\rm rat}(\nspace{C}{n})$. Since
$\mathcal{F}(0)=(f(e)-I_n)(f(e)+I_n)^{-1}$, where $e=(1,\ldots,1)$, is Hermitian and
$I_n-\mathcal{F}(0)=2(f(e)+I_n)^{-1}$ is positive definite, $1\notin\sigma(\mathcal{F}(0))$.
By the maximum principle, $1\notin\sigma(\mathcal{F}(\zeta))$ for every $\zeta\in\mathbb{D}^d$.
Using the same argument as in the proof of the necessity part of Theorem 4.2 in \cite{K-V}, we
first obtain \eqref{eq:Agler-decomp'}
 for $\mathcal{F}$ with rational $N_k\times n$ matrix-valued functions
 \begin{equation}\label{eq:phi_to_theta}
 \theta_k(\zeta)=\frac{1}{1-\zeta_k}\phi_k\left(\frac{1+\zeta_1}{1-\zeta_1},\ldots,
 \frac{1+\zeta_d}{1-\zeta_d}\right)(I_n-\mathcal{F}(\zeta))^{-1},
\quad k=1,\ldots,d
\end{equation}
(clearly, the transformation formulas \eqref{eq:phi_to_theta} and \eqref{eq:theta_to_phi}
are the inverses of each other),
and, in addition,
\begin{equation}\label{eq:2nd_Agler-decomp'}
\mathcal{F}(\omega)^*-\mathcal{F}(\zeta)=\sum_{k=1}^d(\overline{\omega}_k-\zeta_k)\theta_k(\omega)^*\theta_k(\zeta).
\end{equation}
Then observing that the reproducing kernel Hilbert spaces ${\mathcal H}(\theta_k(\omega)^*\theta_k(\zeta))$
are finite-dimensional, that argument produces a finite-dimensional
Givone--Roesser representation \eqref{eq:tf-fin} of $\mathcal{F}$ with the colligation matrix $U$ satisfying
$U^{-1}=U^*=U$.

Conversely, suppose that (3) holds. Then it is easy to verify \eqref{eq:Agler-decomp'} and
\eqref{eq:2nd_Agler-decomp'} with
$$\theta_k(\zeta)=P_k(I_m-AP(\zeta))^{-1}B,\quad k=1,\ldots,d$$
(see \eqref{eq:tf-fin}), where $P_k=\diag[0,\ldots, 0,I_{m_k},0,\ldots,0]$. Then, as in the
proof of the sufficiency part of Theorem 4.2 in \cite{K-V}, we obtain
the two decompositions
\begin{equation}\label{eq:pm-Bessm-decomp'}
f(w)^*\pm f(z)=\sum_{k=1}^d(\overline{w}_k\pm
z_k)\phi_k(w)^*\phi_k(z),
\end{equation}
which together are equivalent to \eqref{eq:Bessm-decomp'}, with
$$\phi_k(z)=\frac{1}{z_k+1}P_k\left(I_m-AP\Big(\frac{z_1-1}{z_1+1},\ldots,\frac{z_d-1}{z_d+1}\Big)\right)^{-1}B(f(z)+I_n)$$
(see \eqref{eq:theta_to_phi}). It is clear that the functions $\phi_k$ have no singularities on $\Pi^k$, that they are rational,
and by the argument in the sufficiency part of Theorem 3.1 in \cite{K-V} they satisfy \eqref{eq:factor-homogen}.
By homogeneity, $\phi_k$ have no singularities in $\Omega_d$.

(2$^\prime$)$\Leftrightarrow$(0) can be proved in the same way as Theorem 2.7 in \cite{K-V} with taking care to
use the additional assumption that all the functions involved are rational and finite matrix-valued. Notice that
implication (0)$\Rightarrow$(2$^\prime$) was
 proved in \cite{Bes} (see also \cite[Theorem 3.2]{Bes1}) under an additional assumption of invertibility of $f(z)$.
\end{proof}

\section{Characterizations of the class $\mathbb{R}\mathcal{B}_d^{n\times n}$} \label{sec:RB_d}

We recall that a $\mathbb{C}^{n\times n}$-valued function $f$ is called real if $f$ is $\iota$-real for $\iota$
being the entrywise complex conjugation on $\nspace{C}{n}$. For this $\iota$, we use the notation
$\mathbb{R}\mathcal{B}_d^{\rm rat}(\nspace{C}{n})$ instead of $\iota\mathbb{R}\mathcal{B}_d^{\rm
rat}(\nspace{C}{n})$.

\begin{thm}\label{thm:rb} Let $f$ be a $\mathbb{C}^{n\times n}$-valued function of $d$ complex variables.
 The following statements are equivalent:
\begin{itemize}
\item[(0)] $f\in\mathbb{R}\mathcal{B}_d^{n\times n}$.
\item[(1)] $f\in\mathbb{R}\mathcal{B}_d^{\rm rat}(\nspace{C}{n})$.
\item[(2)] There exist $\mathbb{C}^{n\times n}$-valued functions
$\Phi_k(w,z)$, $k=1,\ldots, d$,
real rational as functions of $z=(z_1,\ldots,z_d)$ and $\overline{w}=(\overline{w}_1,\ldots,\overline{w}_d)$,
which are  positive  kernels on
   $\Omega_d\times\Omega_d$ satisfying \eqref{eq:ker-homogen}
and \eqref{eq:Bessm-decomp}.
    \item[(2$^{\prime}$)] There exist real rational $\rmat{\mathbb{C}}{N_k}{n}$-valued functions $\phi_k$,
    with some $N_k\in\mathbb{N}$, $k=1,\ldots,d$,
     with no singularities on $\Omega_d$ and satisfy \eqref{eq:factor-homogen} and \eqref{eq:Bessm-decomp'}.
   \item[(3)] There exist $m,m_1,\ldots,m_d\in\mathbb{Z}_+$, with $m=m_1+\cdots+m_d$, and a
matrix
\begin{equation}\label{eq:o-matrix}
U=\begin{bmatrix} A & B\\
C & D
\end{bmatrix}\in\mat{\mathbb{R}}{(m+n)}
\end{equation}
which is symmetric and orthogonal, i.e., $U^{-1}=U^T=U$, where
$$A=[A_{ij}]_{i,j=1,\ldots,d},\quad B=\col_{i=1,\ldots,d}[B_i],\quad C=\row_{j=1,\ldots,d}[C_j]$$
are block matrices with blocks $A_{ij}\in\rmat{\mathbb{R}}{m_i}{m_j}$, $B_i\in\rmat{\mathbb{R}}{m_i}{n}$, and
$C_j\in\rmat{\mathbb{R}}{n}{m_j}$, such that $\mathcal{F}=\mathcal{C}(f)$ (see \eqref{eq:double-Cayley}) has the
representation \eqref{eq:tf-fin}, and $1\notin\sigma(\mathcal{F}(0))$.
     \end{itemize}
\end{thm}
\begin{proof}
We shall follow the same route as in the proof of Theorem \ref{thm:b}, verifying the ``reality" of all functions
and matrices of interest.

(2$^\prime$)$\Rightarrow$(2) is obvious: just define $\Phi_k(w,z)=\phi_k(w)^*\phi_k(z)$, $k=1,\ldots,d$.

(2)$\Rightarrow$(1): It follows from implication (2)$\Rightarrow$(0) of Theorem \ref{thm:bessm-class}, where both
conditions (0) and (2) are modified as indicated in the last part of that theorem, that
$f\in\mathbb{R}\mathcal{B}_d(\mathbb{C}^n)$. Since the kernels $\Phi_k(w,z)$ are rational matrix-valued functions
in both $\overline{w}$ and $z$, the matrix-valued function $f$ in \eqref{eq:Bessm-decomp} is rational.

(1)$\Rightarrow$(2$^\prime$): If (1) holds, then it follows from implication (1)$\Rightarrow$(2$^\prime$) of
Theorem \ref{thm:b} that there exist $N_k\in\mathbb{N}$ and rational $\mathbb{C}^{N_k\times n}$-valued functions
$\phi_k$ with no singularities on $\Omega_d$, $k=1,\ldots,d$, such that \eqref{eq:factor-homogen} and
\eqref{eq:Bessm-decomp'} hold. Moreover, since $f$ is real, \eqref{eq:factor-homogen} and
\eqref{eq:Bessm-decomp'} also hold with $\phi_k$ replaced by $\phi_k^\sharp$, $k=1,\ldots,d$, or with $\phi_k$
replaced by $\widetilde{\phi_k}=\col\Big[\frac{\phi_k+\phi_k^\sharp}{\sqrt{2}},\,
\frac{\phi_k-\phi_k^\sharp}{\sqrt{2}i}\Big]$. It remains to observe that rational $\mathbb{C}^{2N_k\times
n}$-valued functions $\widetilde{\phi_k}$ are real.

(2$^{\prime}$)$\Leftrightarrow$(3): Suppose that (2$^\prime$)
holds. Repeating the construction of matrix $U$ as in the proof of
implication (2$^{\prime}$)$\Rightarrow$(3) of Theorem \ref{thm:b}
and observing that the constructed matrix $U$ is real, we obtain
(3).

Conversely, suppose that (3) holds. Repeating the construction of functions $\phi_k$ as in the proof of
implication (3)$\Rightarrow$(2$^{\prime}$) of Theorem \ref{thm:b} and observing that the constructed functions
$\phi_k$ are real, we obtain (2$^{\prime}$).

(2$^\prime$)$\Leftrightarrow$(0) can be proved in the same way as Theorem 2.7 in \cite{K-V} with taking care to
use the additional assumptions that all the functions involved are real rational matrix-valued and that matrices
$A_k$ in the representation \eqref{eq:long-res}--\eqref{eq:lin-pencil} for $f$ are real. Notice that implication
(0)$\Rightarrow$(2$^\prime$) was proved in \cite{Bes} (see also \cite[Theorem 3.2]{Bes1}) under an additional
assumption of invertibility of $f(z)$.
\end{proof}

\section{Rational Cayley inner Herglotz-Agler-class functions on
${\mathbb D}^{d}$ }  \label{S:CayinD}

 We characterize the rational Cayley inner Herglotz--Agler class
$\mathcal{CIHA}^{\text{rat}}({\mathbb D}^{d}, {\mathbb C}^{n})$
over the polydisk $\nspace{D}{d}$ in the following theorem, which
parallels Theorem \ref{thm:ratinSA} for the rational inner
Schur--Agler class $\mathcal{ISA}_d^{\rm rat}(\nspace{C}{n})$.

\begin{thm}  \label{thm:CayinratHA}
   Let $F$ be a ${\mathbb C}^{n \times n}$-valued function of $d$ complex variables.  The following statements are
   equivalent:
    \begin{enumerate}
    \item[(1)] $F\in\mathcal{CIHA}^{\text{rat}}({\mathbb
    D}^{d}, {\mathbb C}^{n})$.
       \item[(2$^{\prime}$)]  There exist rational
    ${\mathbb C}^{N_{k} \times n}$-valued functions
    $\xi_{k}$, with some $N_k\in\mathbb{N}$, $k=1,\ldots,d$, which have no singularities on $\nspace{D}{d}$ and
    satisfy
    \begin{equation}\label{eq:HA-polydisk-decomp'}
   F(\omega)^{*}  + F(\zeta)  = \sum_{k=1}^{d}(1 - \overline{\omega}_{k} \zeta_{k})
    \xi_{k}(\omega)^{*} \xi_{k}(\zeta).
    \end{equation}
    \item[(3)] There exist $m,m_1,\ldots,m_d\in\mathbb{Z}_+$, with $m=m_1+\cdots+m_d$, a unitary matrix
    $W\in\mathbb{C}^{m\times m}$, a matrix $\beta\in\mathbb{C}^{n\times n}$, and a matrix $V\in\mathbb{C}^{m\times n}$
    such that
    \begin{equation}\label{eq:tf-HA}
   F(\zeta) = \beta + V^{*} ( W - P(\zeta))^{-1} (W + P(\zeta) )V,
    \end{equation}
    where $\beta + \beta^{*} = 0$ and $P(z) = \diag[\zeta_1 I_{m_1},\ldots,\zeta_d I_{m_d}]$.
\end{enumerate}
\end{thm}
\begin{proof}
(1)$\Rightarrow$(3): Represent $F(0)=\beta+\gamma$, where
$\beta=-\beta^*$ and $\gamma=\gamma^*$. Since
$F\in\mathcal{HA}(\nspace{D}{d},\nspace{C}{n})$, the matrix
$\gamma$ is positive semidefinite. Then $\gamma=\delta^*\delta$
with some matrix $\delta\in\mathbb{C}^{r\times n}$ of full row
rank $r$($={\rm rank}\,\gamma$).  We also have that
$F-\beta\in\mathcal{CIHA}^{\rm rat}(\nspace{D}{d},\nspace{C}{n})$.
By the maximum principle, $\ker(F(\zeta)-\beta)=\ker(F(0)-\beta)$
($=\ker\gamma$). Therefore, one can represent $F$ as
$$F(\zeta)=\beta+\delta^*F_+(\zeta)\delta,$$
with $F_+\in\mathcal{CIHA}^{\rm rat}(\nspace{D}{d},\nspace{C}{r})$ satisfying $F_+(0)=I_r$. Define
\begin{equation}\label{eq:Cayley}
\mathcal{F}_+(\zeta)=(F_+(\zeta)-I_r)(F_+(\zeta)+I_r)^{-1}.
\end{equation}
 We have $\mathcal{F}_+\in\mathcal{ISA}_d^{\rm
rat}(\nspace{C}{r})$. By Theorem \ref{thm:ratinSA}, there exist $m$, $m_1$, \ldots, $m_d\in\mathbb{Z}_+$, with
$m=m_1+\cdots+m_d$, and a
unitary matrix $U=\begin{bmatrix} A & B\\
C & D\end{bmatrix}\in\mathbb{C}^{(m+r)\times(m+r)}$ such that
\eqref{eq:tf-fin} holds (with $n$ replaced by $r$ and
$\mathcal{F}$ replaced by $\mathcal{F}_+$). Notice that
$D=\mathcal{F}_+(0)=0$. Therefore, $U$ has the form
\begin{equation}\label{eq:U-decomp}
U=\begin{bmatrix}
A_0 & 0   & 0\\
0   & 0   & B_0\\
0   & C_0 & 0
\end{bmatrix}\colon\,\begin{bmatrix}
\ker C\\
{\rm range}\,C^*\\
\mathbb{C}^r
\end{bmatrix}\to\,\begin{bmatrix}
\ker B^*\\
{\rm range}\,B\\
\mathbb{C}^r
\end{bmatrix}.
\end{equation} In order to obtain a representation \eqref{eq:tf-HA} for $F$, we first obtain a similar representation for
$F_+$ applying the argument from \cite[Pages 63--64]{Ag}. We first
rewrite \eqref{eq:Cayley} as
$$\mathcal{F}_+(\zeta)(F_+(\zeta)+I_r)=F_+(\zeta)-I_r.$$
Together with \eqref{eq:tf-fin}, this is equivalent to
$$\begin{bmatrix}
P(\zeta)A & P(\zeta)B\\
C     &  0
\end{bmatrix}
\begin{bmatrix}
X(\zeta)\\
F_+(\zeta)+I_r
\end{bmatrix}=
\begin{bmatrix}
X(\zeta)\\
F_+(\zeta)-I_r
\end{bmatrix},$$
with $X(\zeta)=(I_r-P(\zeta)A)^{-1}P(\zeta)B(F_+(\zeta)+I_r)$, or to
$$\begin{bmatrix}
X(\zeta)\\
F_+(\zeta)
\end{bmatrix}=\begin{bmatrix}
I_m-P(\zeta)A & -P(\zeta)B\\
-C     &  I_r
\end{bmatrix}^{-1}\begin{bmatrix}
P(\zeta)B\\
I_r
\end{bmatrix}.$$
Using the block-matrix inversion formula (see, e.g.,
\cite[II.5.4]{G}), we obtain
$$\begin{bmatrix}
I_m-P(\zeta)A & -P(\zeta)B\\
-C     &  I_r
\end{bmatrix}^{-1}=\begin{bmatrix}
(I_m-P(\zeta)W^*)^{-1} & (I_m-P(\zeta)W^*)^{-1}P(\zeta)B\\
C(I_m-P(\zeta)W^*)^{-1}     &  C(I_m-P(\zeta)W^*)^{-1}P(\zeta)B+I_r
\end{bmatrix},$$
where $W^*=A+BC$. Taking into account \eqref{eq:U-decomp}, we can rewrite $W^*$ as
$$W^*=\begin{bmatrix}
A_0 & 0\\
0   & B_0C_0
\end{bmatrix}\colon\begin{bmatrix}
\ker C\\
 {\rm range}\,C^*
 \end{bmatrix}\to\begin{bmatrix}
 \ker B^*\\
 {\rm range}\,B
 \end{bmatrix},$$
 and since $A_0\colon\ker C\to \ker B^*$, $B_0\colon\nspace{C}{r}\to{\rm range}\,B$, and $C_0\colon{\rm
 range}\,C_0^*\to\nspace{C}{r}$ are unitary operators, so is $W^*$. Identifying the operator $W\colon
 \nspace{C}{m+r}\to\nspace{C}{m+r}$ with its matrix in the standard basis, we conclude that $W$ is a unitary
 matrix.
We have therefore
$$F_+(\zeta)=2C(I_m-P(\zeta)W^*)^{-1}P(\zeta)B+I_r.$$
Observing that $B^*B=I_r$ and $B^*W^*=C$, we obtain
\begin{multline*}
F_+(\zeta)=2B^*W^*(I_m-P(\zeta)W^*)^{-1}P(\zeta)B+B^*B\\
=2B^*(I_m-W^*P(\zeta))^{-1}W^*P(\zeta)B+B^*B\\
=B^*(I_m-W^*P(\zeta))^{-1}(I_m+W^*P(\zeta))B\\
=B^*(W-P(\zeta))^{-1}(W+P(\zeta))B. \end{multline*} Setting $V=B\delta$, we obtain the desired representation
\eqref{eq:tf-HA} for $F$.

(3)$\Rightarrow$(2$^\prime$): Using \eqref{eq:tf-HA}, one easily
obtains \eqref{eq:HA-polydisk-decomp'} with functions
$$\xi_k(\zeta)=\sqrt{2}P_k(I_m-W^*P(\zeta))^{-1}V$$ (here
$N_k=m_k$)  having the required properties.

(2$^\prime$)$\Rightarrow$(1): Using hereditary functional calculus
as in \cite{Ag}, we obtain from \eqref{eq:HA-polydisk-decomp'}
that
$$F(T)^*+F(T)=\sum_{k=1}^d(I_\mathcal{H}-T_k^*T_k)\xi_k(T)^*\xi_k(T)$$
is a positive semidefinite operator on
$\nspace{C}{n}\otimes\mathcal{H}\cong\mathcal{H}^n$ for every
$d$-tuple $T=(T_1,\ldots,T_d)$ of commuting strict contractions on
a Hilbert space $\mathcal{H}$. It is also obvious from
\eqref{eq:HA-polydisk-decomp'} that $F$ is rational, and has no
singularities on $\nspace{D}{d}$. Finally, the union of
singularity sets for rational matrix-valued functions $\xi_k$,
$k=1,\ldots,d$, inside the unit torus $\nspace{T}{d}$ is of
measure zero (with respect to the Lebesgue measure on
$\nspace{T}{d}$). Hence, $F$ is regular almost everywhere on
$\nspace{T}{d}$, and we see from \eqref{eq:HA-polydisk-decomp'}
that at those regular points $\zeta$ we have
$F(\zeta)^*+F(\zeta)=0$, i.e., $F$ is Cayley inner.
\end{proof}

\section{Rational Cayley inner Herglotz-Agler-class functions on
$\Pi^{d}$}  \label{S:CayinPi}

In this section, we characterize the rational Cayley inner
Herglotz--Agler class $\mathcal{CIHA}^{\text{rat}}(\Pi^d, {\mathbb
C}^{n})$ over the poly-halfplane $\Pi^{d}$. As we mentioned in
Introduction, it can be viewed as a version of the class
${\mathcal B}_{d}^{\rm rat}(\nspace{C}{n})$ with the homogeneity
condition \eqref{eq:homog} dropped. Let us say that $f$ is in the
nonhomogeneous Bessmertny\u{\i} class $\widetilde {\mathcal
B}_{d}^{n \times n}$ if $f$ has a long resolvent representation as
in \eqref{eq:long-res}  and \eqref{eq:lin-pencil} subject to the
conditions
\begin{equation}  \label{eq:NHBes}
    A_{0} = -A_{0}^{*}, \quad A_{k} = A_{k}^{*} \ge 0 \text{ for } k
    = 1, \dots, d.
 \end{equation}
 Then we have the following result.

 \begin{thm}  \label{thm:NHBes}
    Let $f$ be a ${\mathbb C}^{n \times n}$-valued function of $d$ complex variables.  The following statements are
   equivalent:
     \begin{itemize}
     \item[(0)] $f\in\widetilde {\mathcal
B}_{d}^{n \times n}$.
     \item[(1)] $f\in \mathcal{CIHA}^{\text{rat}}(\Pi^{d}, {\mathbb
     C}^{n}$).
         \item[(2$^{\prime}$)] There exist rational ${\mathbb
     C}^{N_{k} \times n}$-valued functions $\phi_{k}$ with some $N_{k} \in {\mathbb N}$,
     $k = 1, \dots, d$, so that
\begin{equation}\label{eq:HA-polyhalfpl-decomp'}
     f(w)^{*} + f(z)  = \sum_{k=1}^{d} (\overline{w}_{k} + z_{k})
     \phi_{k}(w)^{*} \phi_{k}(z).
     \end{equation}

     \item[(3)] $\mathcal{F} = {\mathcal C}(f)\in\mathcal{ISA}^{\rm rat}_{d}({\mathbb C}^{n})$.
      \end{itemize}
      \end{thm}

     \begin{proof}
         The equivalence (1)$\Leftrightarrow$(3) is
         straightforward. The equivalence
(2$^{\prime}$)$\Leftrightarrow$(3)
         follows directly from the equivalence (2$^{\prime}$)$\Leftrightarrow$(3)
                  in Theorem \ref{thm:ratinSA} upon taking double Cayley
         transform and converting \eqref{eq:Agler-decomp'} to
         \eqref{eq:HA-polyhalfpl-decomp'} using
         \eqref{eq:theta_to_phi} or, equivalently,
         \eqref{eq:phi_to_theta}. To complete the proof, it will
         suffice to show implications
         (0)$\Rightarrow$(2$^{\prime}$)and (3)$\Rightarrow$(0).

         (0)$\Rightarrow$(2$^{\prime}$):  This piece is carried
         out in the proof of \cite[Theorem 4.3]{B} for the more general
         nonrational case, where infinite-dimensional
         long-resolvent representations are allowed, i.e., the state space ${\mathbb
         C}^{m}$ is replaced by an infinite-dimensional
         Hilbert space ${\mathcal X}$ (see also a similar argument for the homogeneous case in the proof of
         \cite[Theorem 2.7]{K-V}).
          We here tailor
         the argument presented there for the case of a
         finite-dimensional state space $\cX = {\mathbb C}^{m}$.
         The assumption (0) gives us a long-resolvent
         representation for $f$:
      $$
    f(z)  = A_{11}(z) - A_{12}(z) A_{22}(z)^{-1} A_{21}(z)
    $$
    where
 $$
\begin{bmatrix}  A_{11}(z) & A_{12}(z) \\ A_{21}(z) &   A_{22}(z) \end{bmatrix} = A(z): = A_{0} + z_1A_{1}
    + \cdots + z_dA_{d}
$$
and the coefficients $A_k$ satisfy \eqref{eq:NHBes}.
     We compute
    \begin{align*}
        f(z) & = \begin{bmatrix} I_m &
         -A_{21}(w)^* A_{22}(w)^{*-1}\end{bmatrix}
\begin{bmatrix} A_{11}(z) - A_{12}(z) A_{22}(z)^{-1} A_{21}(z) \\ 0 \end{bmatrix}  \\
    & = \begin{bmatrix} I_m \\
         - A_{22}(w)^{-1}A_{21}(w)\end{bmatrix}^*
    \begin{bmatrix} A_{11}(z) & A_{12}(z) \\ A_{21}(z)
    & A_{22}(z) \end{bmatrix} \begin{bmatrix} I_m \\ -A_{22}(z)^{-1}
    A_{21}(z) \end{bmatrix}\\
    & = \psi(w)^* A(z) \psi(z),
    \end{align*}
 where $\psi(z) : =
\begin{bmatrix} I_m \\ -A_{22}(z)^{-1}
    A_{21}(z) \end{bmatrix}$ is a
    rational $\mathbb{C}^{(m+n)\times n}$-valued function.
    Interchanging the roles of $z$ and $w$, we obtain also
    $$f(w)^*=\psi(w)^*A(w)^*\psi(z).$$
Therefore
    $$ f(w)^{*}+f(z) =\psi(w)^*(A(w)^*+A(z))\psi(z)=
    \sum_{k=1}^{d} (\overline{w}_{k}+z_{k})
    \phi_{k}(w)^*\phi_k(z),
         $$
with $\phi_k(z)=A_k^{1/2}\psi(z)$, $k=1,\ldots,d$.

    (3)$\Rightarrow$(0): Applying the single Cayley transformation
    over the variables to $f$ or the single Cayley transformation
    over the values to $\mathcal{F}$, we obtain
    $F\in\mathcal{CIHA}^{\rm rat}(\nspace{D}{d},\nspace{C}{n})$:
    \begin{equation}\label{eq:single-Cayley}
F(\zeta)=f\left(\frac{1+\zeta_1}{1-\zeta_1},\ldots,\frac{1+\zeta_d}{1-\zeta_d}\right)
=\Big(\mathcal{F}(\zeta)-I_n\Big)^{-1}\Big(\mathcal{F}(\zeta)+I_n\Big).
    \end{equation}
    By Theorem \ref{thm:CayinratHA}, $F$ admits a representation
    \eqref{eq:tf-HA}. Then
    \begin{equation}\label{eq:f_via_M}
f(z)=F\left(\frac{z_1-1}{z_1+1},\ldots,\frac{z_d-1}{z_d+1}\right)=\beta+V^*M(z)V,
    \end{equation}
    where
\begin{multline*}
M(z)=\left(W-P\Big(\frac{z_1-1}{z_1+1},\ldots,\frac{z_d-1}{z_d+1}\Big)\right)^{-1}\left(W+P\Big(
\frac{z_1-1}{z_1+1},\ldots,\frac{z_d-1}{z_d+1}\Big)\right)\\
=\Big(W-(P(z)+I_m)^{-1}(P(z)-I_m)\Big)^{-1}\Big(W+(P(z)+I_m)^{-1}(P(z)-I_m)\Big)\\
=\Big((P(z)+I_m)W-(P(z)-I_m)\Big)^{-1}\Big((P(z)+I_m)W+(P(z)-I_m)\Big)\\
=\Big(P(z)(W-I_m)+(W+I_m)\Big)^{-1}\Big(P(z)(W+I_m)+(W-I_m)\Big).
\end{multline*}
In order to obtain a more detailed representation for $M$ (and
eventually, for $f$), we adopt the idea from \cite{ADY1} of using
a partial Cayley transform of $W$ as follows (the paper
\cite{ADY1} deals with the Nevanlinna--Agler class which can be
obtained from the Herglotz--Agler class by multiplying all the
variables and the function values by $i$). Let
$\mathcal{H}\subset\nspace{C}{m}$ be the eigenspace of $W$
corresponding to the eigenvalue $1$. Then, with respect to the
orthogonal decomposition
$\nspace{C}{m}=\mathcal{H}\oplus\mathcal{H}^\perp$, one has
$$W=\begin{bmatrix}
I_\mathcal{H} & 0\\
0 & W_0
\end{bmatrix},\quad W-I_m=\begin{bmatrix}
0 & 0\\
0 & W_0-I_{{\mathcal H}^\perp}
\end{bmatrix}, \quad W+I_m=\begin{bmatrix}
2I_\mathcal{H} & 0\\
0 & W_0+I_{{\mathcal H}^\perp}
\end{bmatrix},$$
where $W_0-I_{\mathcal{H}^\perp}$ is invertible. We thus obtain
\begin{multline*}
M(z)=\left(P(z)\begin{bmatrix}
0 & 0\\
0 & W_0-I_{{\mathcal H}^\perp}
\end{bmatrix}+\begin{bmatrix}
2I_\mathcal{H} & 0\\
0 & W_0+I_{{\mathcal H}^\perp}
\end{bmatrix}\right)^{-1}\\
\cdot\left(P(z)\begin{bmatrix}
2I_\mathcal{H} & 0\\
0 & W_0+I_{{\mathcal H}^\perp}
\end{bmatrix}+\begin{bmatrix}
0 & 0\\
0 & W_0-I_{{\mathcal H}^\perp}
\end{bmatrix}\right).\end{multline*}
Set
$$\alpha:=(I_{\mathcal{H}^\perp}-W_0)^{-1}(I_{\mathcal{H}^\perp}+W_0).$$
Then
$$W_0=(\alpha-I_{\mathcal{H}^\perp})(\alpha+I_{\mathcal{H}^\perp})^{-1}=I_{\mathcal{H}^\perp}-
2(\alpha+I_{\mathcal{H}^\perp})^{-1},$$ and we can rewrite $M(z)$
first as
\begin{multline*}
M(z) =\begin{bmatrix}
\frac{1}{2}I_\mathcal{H} & 0\\
0 & (W_0-I_{{\mathcal H}^\perp})^{-1}
\end{bmatrix}\\
\cdot\left(P(z)\begin{bmatrix}
0 & 0\\
0 & I_{{\mathcal H}^\perp}
\end{bmatrix}+\begin{bmatrix}
I_\mathcal{H} & 0\\
0 & (W_0+I_{{\mathcal H}^\perp})(W_0-I_{\mathcal{H}^\perp})^{-1}
\end{bmatrix}\right)^{-1}\\
\cdot\left(P(z)\begin{bmatrix}
I_\mathcal{H} & 0\\
0 & (W_0+I_{{\mathcal H}^\perp})(W_0-I_{\mathcal{H}^\perp})^{-1}
\end{bmatrix}+\begin{bmatrix}
0 & 0\\
0 & I_{{\mathcal H}^\perp}
\end{bmatrix}\right)\\
\cdot\begin{bmatrix}
2I_\mathcal{H} & 0\\
0 & W_0-I_{{\mathcal H}^\perp}
\end{bmatrix},\end{multline*}
and then as
\begin{multline}
M(z)=\begin{bmatrix}
I_\mathcal{H} & 0\\
0 & -(\alpha+I_{{\mathcal H}^\perp})
\end{bmatrix}\left(P(z)\begin{bmatrix}
0 & 0\\
0 & I_{{\mathcal H}^\perp}
\end{bmatrix}+\begin{bmatrix}
I_\mathcal{H} & 0\\
0 & -\alpha
\end{bmatrix}\right)^{-1}\\
\cdot\left(P(z)\begin{bmatrix}
I_\mathcal{H} & 0\\
0 & -\alpha
\end{bmatrix}+\begin{bmatrix}
0 & 0\\
0 & I_{{\mathcal H}^\perp}
\end{bmatrix}\right)\begin{bmatrix}
I_\mathcal{H} & 0\\
0 & -(\alpha+I_{{\mathcal H}^\perp})^{-1}
\end{bmatrix}. \label{eq:M_ADY1}
\end{multline}
We observe that since the operator $W_0$ is unitary, its Cayley
transform $\alpha$ is skew-adjoint, i.e., $\alpha^*=-\alpha$.
Therefore, we can rewrite \eqref{eq:M_ADY1} as
$$M(z)=\begin{bmatrix}
I_\mathcal{H} & 0\\
0 & -(\alpha+I_{{\mathcal H}^\perp})
\end{bmatrix}N(z)\begin{bmatrix}
I_\mathcal{H} & 0\\
0 & -(\alpha+I_{{\mathcal H}^\perp})^*
\end{bmatrix},$$ where
\begin{multline*}
N(z)=\left(P(z)\begin{bmatrix}
0 & 0\\
0 & I_{{\mathcal H}^\perp}
\end{bmatrix}+\begin{bmatrix}
I_\mathcal{H} & 0\\
0 & -\alpha
\end{bmatrix}\right)^{-1}\\
\cdot\left(P(z)\begin{bmatrix}
I_\mathcal{H} & 0\\
0 & -\alpha
\end{bmatrix}+\begin{bmatrix}
0 & 0\\
0 & I_{{\mathcal H}^\perp}
\end{bmatrix}\right)\begin{bmatrix}
I_\mathcal{H} & 0\\
0 & (I_{{\mathcal H}^\perp}-\alpha^2)^{-1}
\end{bmatrix}
\end{multline*}
and we use that
$-(\alpha+I_{\mathcal{H}^\perp})^*=\alpha-I_{\mathcal{H}^\perp}$.
Writing
$$P(z)=\begin{bmatrix}
P_{11}(z) & P_{12}(z)\\
P_{21}(z) & P_{22}(z)
\end{bmatrix},$$ we compute
\begin{multline*}
\left(P(z)\begin{bmatrix}
0 & 0\\
0 & I_{{\mathcal H}^\perp}
\end{bmatrix}+\begin{bmatrix}
I_\mathcal{H} & 0\\
0 & -\alpha
\end{bmatrix}\right)^{-1}=\begin{bmatrix}
I_\mathcal{H} & P_{12}(z)\\
0 & P_{22}(z)-\alpha
\end{bmatrix}^{-1}\\
=\begin{bmatrix}
I_\mathcal{H} & -P_{12}(z)(P_{22}(z)-\alpha)^{-1}\\
0 & (P_{22}(z)-\alpha)^{-1}
\end{bmatrix},
\end{multline*}
\begin{multline*}
\left(P(z)\begin{bmatrix}
I_\mathcal{H} & 0\\
0 & -\alpha
\end{bmatrix}+\begin{bmatrix}
0 & 0\\
0 & I_{{\mathcal H}^\perp}
\end{bmatrix}\right)\begin{bmatrix}
I_\mathcal{H} & 0\\
0 & (I_{{\mathcal H}^\perp}-\alpha^2)^{-1}
\end{bmatrix}\\
=\begin{bmatrix} P_{11}(z) &
P_{12}(z)J\\
P_{21}(z) &
(I_{\mathcal{H}^\perp}-P_{22}(z)\alpha)(I_{\mathcal{H}^\perp}-\alpha^2)^{-1}
\end{bmatrix},
\end{multline*}
where $J:=-\alpha(I_{\mathcal{H}^\perp}-\alpha^2)^{-1}$. (Notice
that $J^*=-J$.) Next, writing
$$I_{\mathcal{H}^\perp}-P_{22}(z)\alpha=I_{\mathcal{H}^\perp}-\alpha^2+\alpha^2-P_{22}(z)\alpha,$$
we obtain
\begin{multline*}
(I_{\mathcal{H}^\perp}-P_{22}(z)\alpha)(I_{\mathcal{H}^\perp}-\alpha^2)^{-1}=
I_{\mathcal{H}^\perp}-(P_{22}(z)-\alpha)\alpha(I_{\mathcal{H}^\perp}-\alpha^2)^{-1}\\
=I_{\mathcal{H}^\perp}+(P_{22}(z)-\alpha)J.
\end{multline*}
Thus
\begin{multline*}
N(z)=\begin{bmatrix}
I_\mathcal{H} & -P_{12}(z)(P_{22}(z)-\alpha)^{-1}\\
0 & (P_{22}(z)-\alpha)^{-1}
\end{bmatrix}\begin{bmatrix} P_{11}(z) &
P_{12}(z)J\\
P_{21}(z) & I_{\mathcal{H}^\perp}+(P_{22}(z)-\alpha)J
\end{bmatrix}\\
=\begin{bmatrix}
P_{11}(z)-P_{12}(z)(P_{22}(z)-\alpha)^{-1}P_{21}(z) &
-P_{12}(z)(P_{22}(z)-\alpha)^{-1}\\
(P_{22}(z)-\alpha)^{-1}P_{21}(z) & (P_{22}(z)-\alpha)^{-1}+J
\end{bmatrix}\\
=\begin{bmatrix} P_{11}(z) & 0
\\
0 & J
\end{bmatrix}-\begin{bmatrix} P_{12}(z) \\
-I_{\mathcal{H}^\perp}
\end{bmatrix}(P_{22}(z)-\alpha)^{-1}\begin{bmatrix} P_{21}(z) & I_{\mathcal{H}^\perp}
\end{bmatrix}.
\end{multline*}
Consequently, \begin{multline} M(z) =\begin{bmatrix} P_{11}(z) & 0
\\
0 &
(\alpha+I_{\mathcal{H}^\perp})J(\alpha+I_{\mathcal{H}^\perp})^*
\end{bmatrix}\\
-\begin{bmatrix} P_{12}(z) \\
\alpha+I_{\mathcal{H}^\perp}
\end{bmatrix}(P_{22}(z)-\alpha)^{-1}\begin{bmatrix} P_{21}(z) &
-(\alpha+I_{\mathcal{H}^\perp})^*
\end{bmatrix},\label{eq:Bessm-M}
\end{multline}
and
\begin{multline*}
f(z) =\beta+V^*\begin{bmatrix} P_{11}(z) & 0
\\
0 &
(\alpha+I_{\mathcal{H}^\perp})J(\alpha+I_{\mathcal{H}^\perp})^*
\end{bmatrix}V\\
-V^*\begin{bmatrix} P_{12}(z) \\
\alpha+I_{\mathcal{H}^\perp}
\end{bmatrix}(P_{22}(z)-\alpha)^{-1}\begin{bmatrix} P_{21}(z) &
-(\alpha+I_{\mathcal{H}^\perp})^*
\end{bmatrix}V.
\end{multline*}
Representing $$V=\begin{bmatrix}
 V_1\\ V_2\end{bmatrix}\colon\nspace{C}{n}\to
 \mathcal{H}\oplus\mathcal{H}^\perp=\nspace{C}{m},$$
 we obtain
 \begin{multline}
f(z)
=\beta+V_1^*P_{11}(z)V_1+V_2^*(\alpha+I_{\mathcal{H}^\perp})J(\alpha+I_{\mathcal{H}^\perp})^*V_2\\
-(V_1^*
P_{12}(z)+V_2^*(\alpha+I_{\mathcal{H}^\perp}))(P_{22}(z)-\alpha)^{-1}(P_{21}(z)V_1-(\alpha+I_{\mathcal{H}^\perp})^*
V_2).\label{eq:Bessm-f}
\end{multline}
In other words, we have obtained a long resolvent representation
\eqref{eq:long-res} for $f$, with
\begin{eqnarray*}
A_{11}(z) &=& \beta+V_2^*(\alpha+I_{\mathcal{H}^\perp})J(\alpha+I_{\mathcal{H}^\perp})^*V_2+V_1^*P_{11}(z)V_1,\\
A_{12}(z) &=&  V_2^*(\alpha+I_{\mathcal{H}^\perp})+V_1^* P_{12}(z),\\
A_{21}(z) &=& -(\alpha+I_{\mathcal{H}^\perp})^* V_2+P_{21}(z)V_1,\\
A_{22}(z) &=& -\alpha+P_{22}(z).
\end{eqnarray*}
The linear pencil
$$A(z)=\begin{bmatrix}
A_{11}(z) & A_{12}(z)\\
A_{21}(z) & A_{22}(z)
\end{bmatrix}=A_0+z_1A_1+\cdots+z_dA_d$$
has the coefficients \begin{align}\label{eq:0-coef} A_0 &=\begin{bmatrix}
\beta+V_2^*(\alpha+I_{\mathcal{H}^\perp})J(\alpha+I_{\mathcal{H}^\perp})^*V_2
& V_2^*(\alpha+I_{\mathcal{H}^\perp})\\
-(\alpha+I_{\mathcal{H}^\perp})^* V_2 & -\alpha
\end{bmatrix},\\
A_k &=\begin{bmatrix}
V_1^*(P_k)_{11}V_1 & V_1^*(P_k)_{12}\\
(P_k)_{21}V_1 & (P_k)_{22}
\end{bmatrix},\label{eq:k-coefs}
\end{align}
where $P_k=\diag[0,\ldots,0,I_{m_k},0,\ldots,0]$, $k=1,\ldots,d$.
It is easy to check that the matrices $A_0$, \ldots, $A_d$ satisfy
\eqref{eq:NHBes}. We conclude that
$f\in\widetilde{\mathcal{B}}_d^{n\times n}$.
\end{proof}

\begin{rem}\label{rem:ADY1-vs-us}
It is not obvious from \eqref{eq:M_ADY1} that $M(z)$ and,
therefore, $f$ as in \eqref{eq:f_via_M} are Herglotz--Agler
functions on $\Pi^d$; see (3.1) and (4.3) in \cite{ADY1} for the
Nevanlinna-Agler-class version of $M$ and $f$. It takes several
pages in \cite{ADY1} (see Propositions 3.4 and 3.5 in there) to
show that $M$ is a Nevanlinna--Agler function over the upper
poly-halfplane. Our representation \eqref{eq:Bessm-M} for $M$ and
then our representation \eqref{eq:Bessm-f} for $f$ allow us to see
the inclusion of these functions to the corresponding
matrix-valued Herglotz--Agler classes over $\Pi^d$ immediately.
Indeed, the direction (0)$\Rightarrow$(1) in Theorem
\ref{thm:NHBes} follows from the fact that a linear pencil $A(z)$
as in \eqref{eq:lin-pencil} with $A_0^*=-A_0$ and $A_k=A_k^*\ge
0$, $k=1,\ldots,d$, is a rational Cayley inner Herglotz--Agler
function, and taking a Schur complement of $A(z)$ preserves this
property (it just changes the matrix size for the function
values).
\end{rem}

\begin{rem}\label{rem:b-char}
Theorem \ref{thm:NHBes} is a generalization of Theorem
\ref{thm:b}, or we can view Theorem \ref{thm:b} as a
specialization of Theorem \ref{thm:NHBes} to the Bessmertny\u{\i} class
$\mathcal{B}_d^{n\times n}$. Thus, $\mathcal{B}_d^{n\times n}$ can
be characterized as \begin{enumerate}
    \item[(0)] the subclass in $\widetilde{\mathcal{B}}_d^{n\times n}$
    consisting of functions with a long resolvent representation
    satisfying $A_0=0$, or
    \item[(1)] (in view of Theorem \ref{thm:bessm-class}) the subclass in $\mathcal{CIHA}^{\rm
    rat}(\Pi^d,\nspace{C}{n})$ satisfying the homogeneity
    condition \eqref{eq:homog}, or
    \item[(2$^\prime$)] the subclass of functions satisfying
    condition (2$^\prime$) of Theorem \ref{thm:NHBes} with the
    additional property that, along with the decomposition
    \eqref{eq:HA-polyhalfpl-decomp'}, they have the decomposition
    obtained from \eqref{eq:HA-polyhalfpl-decomp'} by replacing
    pluses by minuses, i.e., that \eqref{eq:pm-Bessm-decomp'}
    holds, or
    \item[(3)] the subclass of functions satisfying condition (3)
    of Theorem  \ref{thm:NHBes} with the
    additional property that the matrix $U$ is Hermitian.
\end{enumerate}
\end{rem}

Applying the single Cayley transformation
    over the variables to a function $f\in\mathcal{CIHA}^{\rm rat}(\Pi^d,\nspace{C}{n})$
     or the single Cayley transformation
    over the values to $\mathcal{F}=\mathcal{C}(f)$ (see \eqref{eq:single-Cayley}), we obtain
    $F\in\mathcal{CIHA}^{\rm rat}(\nspace{D}{d},\nspace{C}{n})$.
    By Theorem \ref{thm:CayinratHA}, $F$ admits a representation
    \eqref{eq:tf-HA}. The following theorem specializes
    \eqref{eq:tf-HA} to the case where $F$ is the single Cayley
    transform of a function $f\in\mathcal{B}_d^{n\times n}$.

\begin{thm}\label{thm:Bessm-vs-CIHA-polydisk}
 Let $f$ be a ${\mathbb C}^{n \times n}$-valued function of $d$ complex
 variables. Then $f\in\mathcal{B}_d^{n\times n}$ if and only if
 the function
\begin{equation}\label{eq:f-to-F}
F(\zeta)=f\left(\frac{1+\zeta_1}{1-\zeta_1},\ldots,\frac{1+\zeta_d}{1-\zeta_d}\right)
\end{equation}
 satisfies the condition (3) of Theorem \ref{thm:CayinratHA} with
 the additional properties
 \begin{enumerate}
    \item[(i)] $\beta=0$;
    \item[(ii)] $W=W^*$;
    \item[(iii)] ${\rm range}\,V\subseteq\mathcal{H}$, where
    $\mathcal{H}\subseteq\nspace{C}{m}$ is the eigenspace of $W$
    corresponding to the eigenvalue $1$.
 \end{enumerate}
\end{thm}

For the proof of Theorem \ref{thm:Bessm-vs-CIHA-polydisk}, we will
need the following lemma.

\begin{lem}\label{lem:Bessm-normalized}
Let $f$ be a ${\mathbb C}^{n \times n}$-valued function of $d$
complex
 variables. Then $f\in\mathcal{B}_d^{n\times n}$ if and only if
 there exist a matrix $\delta\in\mathbb{C}^{r\times n}$ of full
 row rank $r$ ($={\rm rank}\,f(e)$, where $e=(1,\ldots,1)$) and a
 function $f_+\in\mathcal{B}_d^{r\times r}$ satisfying
 $f_+(e)=I_r$, such that $$f(z)=\delta^*f_+(z)\delta.$$
\end{lem}
\begin{proof}
Suppose that $f_+\in\mathcal{B}_d^{r\times r}$, $f_+(e)=I_r$, and $f(z)=\delta^*f_+(z)\delta$ for some
$\delta\in\mathbb{C}^{r\times n}$. If $f_+$ has a long resolvent representation
\eqref{eq:long-res}--\eqref{eq:lin-pencil} (with $n$ replaced by $r$) with the coefficients $A_0=0$ and
$A_k=A_k\ge 0$, $k=1,\ldots,d$, then $f$ has a long resolvent representation with the coefficients $A_0=0$,
$$\begin{bmatrix}
\delta^* & 0\\
0 & I_m\end{bmatrix}A_k\begin{bmatrix}
\delta & 0\\
0 & I_m\end{bmatrix}\geq 0,\quad k=1,\ldots,d,$$ i.e., $f\in\mathcal{B}_d^{n\times n}$.

Conversely, suppose that $f\in\mathcal{B}_d^{n\times n}$. Let $\mathcal{X}:=\ker f(e)\subseteq\nspace{C}{n}$, so
that $\nspace{C}{n}=\mathcal{X}\oplus \mathcal{X}^\perp$. By the maximum principle, $\ker f(z)=\ker
f(e)=\mathcal{X}$ for every $z\in\Omega_d$ and, thus, for every $z$ a regular point of $f$. Let
$r:=\dim\mathcal{X}^\perp$ and let $\kappa\colon\nspace{C}{r}\to\nspace{C}{n}$ be an isometry with ${\rm
range}\,\kappa=\mathcal{X}^\perp$. Applying the argument as in the preceding paragraph, we obtain that
$$\widetilde{f}:=\kappa^*f(z)\kappa\in\mathcal{B}_d^{r\times r}.$$ Clearly $\widetilde{f}(e)=\widetilde{f}(e)^*>0$. Define
$$f_+(z)=\widetilde{f}(e)^{-1/2}\widetilde{f}(z)\widetilde{f}(e)^{-1/2}.$$
By the same argument as above, $f_+\in\mathcal{B}_d^{r\times r}$. We then have $f(z)=\delta^*f_+(z)\delta$ as
required, with $$\delta=\widetilde{f}(e)^{1/2}\kappa^*.$$
\end{proof}
\begin{proof}[Proof of Theorem \ref{thm:Bessm-vs-CIHA-polydisk}]
Suppose that $F$
 satisfies the condition (3) of Theorem \ref{thm:CayinratHA} with
 the additional properties (i)--(iii). Arguing as in the proof of (3)$\Rightarrow$(0) of Theorem
 \ref{thm:NHBes}, we obtain the following. First, we have $\beta=0$. Second, since the matrix $W$ is unitary and
 Hermitian, it has only two eigenvalues, $1$ and $-1$, therefore $W_0=-I_{\mathcal{H}^\perp}$ and $\alpha=J=0$.
 Third, since ${\rm range}\,V\subseteq\mathcal{H}$, we have $V_2=0$. Therefore, $f$ has a long resolvent
 representation \eqref{eq:long-res} with the coefficient matrices $A_0=0$ and $A_k\ge 0$, $k=1,\ldots,d$
 (see \eqref{eq:0-coef}--\eqref{eq:k-coefs}), i.e., $f\in\mathcal{B}_d^{n\times n}$.

Conversely, suppose that $f\in\mathcal{B}_d^{n\times n}$. Then by Lemma \ref{lem:Bessm-normalized} we have that
$f(z)=\delta^*f_+(z)\delta$, with some $\delta\in\mathbb{C}^{r\times n}$ of full
 row rank $r$ ($={\rm rank}\,f(e)$) and a
 function $f_+\in\mathcal{B}_d^{r\times r}$ satisfying
 $f_+(e)=I_r$, such that $f(z)=\delta^*f_+(z)\delta.$
Applying the Cayley transform over variables as in \eqref{eq:f-to-F}, we obtain
$$F(\zeta)=\delta^*F_+(\zeta)\delta,$$
with $F_+\in\mathcal{CIHA}(\nspace{D}{d},\nspace{C}{r})$ satisfying $F(0)=I_r$.  Arguing as in the proof of
(1)$\Rightarrow$(3) of Theorem \ref{thm:CayinratHA}, we obtain the following.  First, we have $\beta=0$. Second,
the function $\mathcal{F}_+$ in \eqref{eq:Cayley} can be represented as $\mathcal{F}_+=\mathcal{C}(f_+)$, so by
Theorem \ref{thm:b} the unitary matrix $U$ in a transfer-function realization \eqref{eq:tf-fin} for
$\mathcal{F}_+$ can be chosen Hermitian. Therefore, $C=B^*$ and we have \eqref{eq:U-decomp} with $A_0=A_0^*$ and
$C_0=B_0^*$. This, in turn, implies that
$$W=W^*=W^{-1}=\begin{bmatrix}
A_0 & 0\\
0 & I_{{\rm range}\, B}
\end{bmatrix}$$
with respect to the decomposition $\nspace{C}{m}=\ker B^*\oplus{\rm range}\, B$. Clearly, we have that ${\rm
range}\, B\subseteq\mathcal{H}$, where $\mathcal{H}$ is the eigenspace of $W$ corresponding to the eignevalue 1.
Then ${\rm range}\, V\subseteq\mathcal{H}$, where $V=B\delta$. Finally, we obtain the representation
\eqref{eq:tf-HA} for $F$ as desired.
\end{proof}

\section{Some remarks}

\subsection{} The results of Section \eqref{sec:RB_d} can be easily extended to the real rational Cayley inner
Herglotz--Agler class over $\Pi^d$, in view of possible applications in electrical engineering. The techniques
developed in this paper  would suffice for a proof, which we leave to the reader as an exercise.

\subsection{} In addition to the rational inner / Cayley inner Schur--Agler and Herglotz--Agler classes over the polydisk
$\nspace{D}{d}$ and
 the Herglotz--Agler class over the right poly-halfplane $\Pi^d$, one can also consider the rational inner Schur--Agler
  class
 over $\Pi^d$ and obtain analogues of Theorems \ref{thm:ratinSA}, \ref{thm:CayinratHA}, and \ref{eq:NHBes}. It is
 also possible to describe the image of the Bessmertny\u{\i} class $\mathcal{B}_d^{n\times n}$ in the latter class
 under the Cayley transform over the function values, similarly to part (3) of Theorem \ref{thm:b} and to Theorem
 \ref{thm:Bessm-vs-CIHA-polydisk}. These ingredients were unnecessary in our analysis and appeared somewhat isolated,
  so we left them aside.

\subsection{} In \cite[Example 5.1]{GKVW}, explicit examples of rational inner functions
on $\nspace{D}{d}$ which are not in the Schur--Agler class over
$\nspace{D}{d}$ were constructed. After applying appropriate
linear-fractional changes of variable to domains and ranges and
combinations thereof to those examples, one can obtain explicit
examples of rational Cayley inner functions on ${\mathbb D}^{d}$
which are not in the Herglotz--Agler class over ${\mathbb D}^{d}$,
of rational inner functions on $\Pi^{d}$ which are not in the
Schur--Agler class over $\Pi^{d}$, and of rational Cayley inner
functions on $\Pi^{d}$ which are not in the Herglotz--Agler class
over $\Pi^{d}$. However, the corresponding question for the
Bessmertny\u{\i} class remains unresolved, as already mentioned in the
Introduction: can there be a function $f\in\mathcal{P}_d^{n\times
n}$ which is not in $\mathcal{B}_d^{n\times n}$?

\end{document}